\documentclass[a4paper, 10pt, reqno]{amsart}
\usepackage[english]{babel}
\usepackage{amsmath}
\usepackage{amssymb}
\usepackage{amsfonts}
\usepackage{amsthm}
\usepackage{mathrsfs}
\usepackage{IEEEtrantools}
\usepackage[shortalphabetic]{amsrefs}
\usepackage{setspace}
\usepackage{fullpage}
\usepackage{titlesec}

\titleformat{\section}[hang]{\scshape\center}{\thesection.}{3pt}{} 

\titleformat{\subsection}[runin]{\bf\scshape}{\thesubsection.}{0pt}{}

\usepackage{tikz}
\usepackage{tikz-cd}

\allowdisplaybreaks

\theoremstyle{plain}
\newtheorem{theorem}{Theorem}[subsection]
\newtheorem{lemma}[theorem]{Lemma}
\newtheorem{proposition}[theorem]{Proposition}

\theoremstyle{definition}

\newcommand{\lie}{\mathfrak{g}}
\newcommand{\lieh}{\mathfrak{h}}
\newcommand{\ol}{\overline}
\newcommand{\CC}{\mathbb{C}}

\newcommand{\ZZ}{\mathbb{Z}}

\begin{document}

\title{Deformed Poisson W-algebras of Type A}
\author{Lachlan Walker}
\address{Department of Mathematics\\
University of York \\
Heslington \\ 
York \\ 
YO10 5DD \\
United Kingdom}
\email{lachlan.walker@cantab.net}

\begin{abstract}
For the algebraic group $SL_{l+1}(\mathbb{C})$ we describe a system of positive roots associated to conjugacy classes in its Weyl group. Using this we explicitly describe the algebra of regular functions on certain transverse slices to conjugacy classes in $SL_{l+1}(\mathbb{C})$ as a polynomial algebra of invariants. These may be viewed as an algebraic group analogue of certain graded parabolic invariants that generate the (graded) W-algebra in type A.
\end{abstract}

\maketitle

\section{Introduction}
\subsection{}
Let $G$ be a complex semisimple connected algebraic group with Lie algebra $\lie$. The Jacobson-Morozov theorem states that any non-zero nilpotent element $e \in \lie$ can be embedded into an $\mathfrak{sl}_{2}$-triple $(e,f,h)$. Fix such a triple and let $\mathfrak{z}(f)$ denote the centraliser of $f$ in $\lie$. The affine space $\mathcal{S}_{e} = e+ \mathfrak{z}(f)$ is called the Slodowy slice and is transverse to the adjoint orbit in $\lie$ containing $e$ at the point $e$. They are named after P. Slodowy who studied the singularities of the adjoint quotient map $\delta_{\lie} : \lie \rightarrow \mathfrak{h}/W$ induced from the inclusion $\mathbb{C}[\lieh]^{W} \rightarrow \mathbb{C}[\lie]$, where $\mathbb{C}[\lieh]^{W} \simeq \mathbb{C}[\lie]^{G}$, by restricting $\delta_{\lie}$ to the spaces $\mathcal{S}_{e}$; see \cite{Slo}. To the nilpotent element $e$ one can construct an associative algebra $U(\lie,e)$ called the W-algebra of the pair $(\lie,e)$; see for example \cite{BK}, \cite{GanGin}, \cite{Wang}. Using a special filtration of $U(\lie,e)$, called the Kazhdan filtration, it was shown in \cite{GanGin} that there exists an isomorphism of graded Poisson algebras
$$
\text{gr}U(\lie,e) \simeq \mathbb{C}[\mathcal{S}_{e}]
$$
where we view $\mathcal{S}_{e} \subseteq \lie^{*}$ using the isomorphism $\lie \simeq \lie^{*}$ induced by the Killing form. \\
The Poisson structure on $\mathcal{S}_{e}$ arises naturally by means of a Hamiltonian reduction of the Poisson manifold $\lie^{*}$ by the action of a unipotent subgroup $M \subseteq G$. Indeed, the map $\text{ad}h : \lie \rightarrow \lie$ gives a decomposition
$$
\lie = \bigoplus_{j \in \mathbb{Z}} \lie(j)
$$
of $\lie$ into eigenspaces $\lie(j) = \{ x \in \lie : \text{ad}h(x) = jx \}$. Use the Killing form to identify the element $e \in \lie$ with $\chi \in \lie^{*}$. The subspace $\lie(-1)$ can be equipped with a non-degenerate, skew-symmetric bilinear form $\omega$ defined by $\omega(x,y) = \chi([x,y])$ for all $x,y \in \lie(-1)$. Fix a Lagrangian subspace $\mathcal{L} \subset \lie(-1)$ with respect to this form and let $\mathfrak{m} = \mathcal{L} \oplus \bigoplus_{j \leq -2} \lie(j)$. Denote by $M$ the unipotent subgroup of $G$ with Lie algebra $\mathfrak{m}$. The coadjoint action of $G$ on $\lie^{*}$ is Hamiltonian and restricts to an action of $M$ on $\lie^{*}$ with moment map $\mu : \lie^{*} \rightarrow \mathfrak{m}^{*}$. In \cite{GanGin} and \cite{Kost} it was shown that the reduced Poisson manifold $\mu^{-1}(\chi|_{\mathfrak{m}})/M$ can be identified with the Slodowy slice $\mathcal{S}_{e}$. This identification is known as the Kostant cross-section theorem and it allows one to view $\mathcal{S}_{e}$ as a Poisson submanifold of the quotient $\lie^{*}/M$, in particular, 
$$
\mathbb{C}[\mathcal{S}_{e}] \simeq \mathbb{C}[\mu^{-1}(\chi|_{\mathfrak{m}})]^{M}
$$
is an isomorphism of Poisson algebras. \\

\subsection{}
Let $W$ denote the Weyl group of $G$. In \cite{Sev} an algebraic group analogue of $\mathcal{S}_{e}$ was constructed in which these new slices depend on Weyl group elements $s \in W$. They are transverse to the conjugacy class containing $s^{-1} \in G$ at the point $s^{-1}$, where $s \in G$ is a representative of the Weyl group element $s \in W$ of the same letter, and for which an analogue of the Kostant cross-section theorem holds. One constructs these slices by associating to $s$ a parabolic subgroup $Q \subseteq G$ such that the semisimple part of its Levi factor $L$ is contained in the centraliser of the element $s \in G$. The slices are closed and are of the form $\Sigma_{s} = N_{s}Zs^{-1}$ where $N_{s} = \{ n \in N : sns^{-1} \in N_{-} \}$ and $N, N_{-}$ are the unipotent and opposite unipotent radicals of $Q$, respectively, and $Z = \{ z \in L : szs^{-1} = z \}$. It is also shown in \cite{Sev} that there is an isomorphism of varieties between the slice $\Sigma_{s}$ and the quotient $NZs^{-1}N/N$ with respect to the action of $N$ on $NZs^{-1}N$ by conjugation, where $NZs^{-1}N$ is also closed. This result is an algebraic group analogue of the Kostant cross-section theorem and in particular gives rise to an isomorphism 
$$
\mathbb{C}[\Sigma_{s}] \simeq \mathbb{C}[NZs^{-1}N]^{N}
$$ 
of algebras. In the context of Poisson geometry $G$ can be equipped with a particular Poisson structure which makes it a Poisson manifold, denoted $G_{*}$. The algebra of regular functions on $NZs^{-1}N$ acquires a Poisson structure from the quotient $\mathbb{C}[G_{*}]/I \simeq \mathbb{C}[NZs^{-1}N]$ where $I$ is the vanishing ideal of $NZs^{-1}N$. This equips $\mathbb{C}[\Sigma_{s}]$ with a Poisson structure via the isomorphism $\mathbb{C}[\Sigma_{s}] \simeq \mathbb{C}[NZs^{-1}N]^{N} \subseteq \mathbb{C}[NZs^{-1}N]$ and the algebra of regular functions on $\Sigma_{s}$ is thus a Poisson algebra called the deformed Poisson W-algebra or just q-W-algebra. \\

\subsection{}
In the present article we aim to exploit this Kostant cross-section analogue to find $N$-invariant functions on the subvariety $NZ's^{-1}N$ in the case $G = SL_{l+1}(\mathbb{C})$ where $Z' \subseteq Z$ is (a certain conjugate of) the big cell in $Z$. The simplest case occurs when $s \in W$ is Coxeter, that is, a product of the simple reflections in $W$. Here, one obtains the slice $\Sigma_{s} = N_{s}s^{-1}$ which is transverse to the set of conjugacy classes of regular elements in $G$ and was first constructed in \cite{Stein}. An inductive method for computing the fundamental characters restricted to $N_{s}s^{-1}$ which appears in \cite{Stein} can be used to show that a set of generators for $\mathbb{C}[Ns^{-1}N]^{N}$ may be described in terms of these fundamental characters. However, for a general Weyl group element the number of generators needed to describe $\mathbb{C}[NZs^{-1}N]^{N}$ is larger than in the Coxeter case and the task of constructing a set of generators for this algebra of invariants becomes more technical. \\

\noindent\textbf{Acknowledgment}. The author would like to thank A. Sevastyanov for many useful discussions on this subject and the support from EPSRC grant EP/N023919. \\

\section{Notation}
\subsection{}
We fix some standard notation used throughout this text. The set of integers, real numbers and complex numbers are denoted by the usual letters $\mathbb{Z}$, $\mathbb{R}$ and $\mathbb{C}$, respectively. Moreover, for any element $n \in \mathbb{Z}$ denote by $\mathbb{Z}_{\geq n}$ and $\mathbb{Z}_{> n}$ the subsets of $\mathbb{Z}$ defined by $\mathbb{Z}_{\geq n} = \{ x \in \mathbb{Z} : x \geq n \}$ and $\mathbb{Z}_{> n} = \{ x \in \mathbb{Z} : x > n \}$. Let $G$ be a complex connected finite-dimensional semisimple algebraic group such that $H$ is a maximal torus in $G$. Let the Lie algebras of $H$ and $G$ be denoted by the gothic letters $\lieh$ and $\lie$, respectively. The root system of the pair $(\lie,\lieh)$ is denoted by $\Delta$ and $\Lambda = \{ \alpha_{1}, \ldots, \alpha_{l} \}$ will denote the system of simple roots where $l = \text{rank}\ \lie$. Write $\lie = \lieh \oplus \bigoplus_{\alpha \in \Delta} \lie_{\alpha}$ for the root decomposition of $\lie$ into root subspaces $\lie_{\alpha}$. Denote by $\Delta_{+}$ the set of positive roots in $\Delta$ with respect to $\Lambda$ and write $\Delta_{-}=-\Delta_{+}$. The Weyl group of the root system $\Delta$ is denoted $W$ and is generated by the simple reflections $s_{\alpha_{1}}, \ldots, s_{\alpha_{l}}$. \\

\section{Transverse Slices in Algebraic Groups}
\subsection{} \label{sect: 3.1}
We begin by outlining a general procedure for constructing transverse slices to conjugacy classes in $G$. These slices depend on conjugacy classes in $W$ and are constructed with the help of certain parabolics associated to elements in $W$. This outline follows the one given in \cite{Sev}. \\
Fix an element $s \in W$. This acts on the real span of simple coroots, denoted $\lieh_{\mathbb{R}}$, as an orthogonal transformation and decomposes into a direct sum of $s$-invariant subspaces
$$
\lieh_{\mathbb{R}} = \bigoplus_{i=0}^{K} \lieh_{i}
$$  
where we assume $\lieh_{0}$ is the subspace of elements fixed by the action of $s$. The subspaces $\lieh_{i}$, for $1 \leq i \leq K$, are either one or two dimensional and $s$ acts on $\lieh_{i}$ as rotation by $\theta_{i}$, where $0 < \theta_{i} \leq \pi$. As $s$ has finite order $\theta_{i} = 2\pi / m_{i}$, for some $m_{i} \in \mathbb{Z}_{\geq0}$. Assume that the non-fixed subspaces in the above decomposition are ordered such that $\lieh_{i}$ is written before $\lieh_{j}$ if $\theta_{i} > \theta_{j}$. For each $i = 0, \ldots , K$ choose $h_{i} \in \lieh_{i}$ such that $h_{i}(\alpha) \neq 0$ for any root $\alpha \in \Delta$ which is not orthogonal to the $s$-invariant subspace $\lieh_{i}$ with respect to the Killing form. Define 
$$
\ol\Delta_{i} = \{ \alpha \in \Delta : h_{i}(\alpha) \neq 0, \ h_{j}(\alpha) = 0, \ j>i \}.
$$
Note that some of these subsets may be empty. Let $\{\ol\Delta_{i_{0}}, \ldots, \ol\Delta_{i_{M}}\} = \{\ol\Delta_{i} : \ol\Delta_{i} \neq \emptyset\}$ where $i_{0} = 0$ if $\ol{\Delta}_{i_{0}} \neq \emptyset$ and we assume $i_{j}<i_{k}$ if and only if $j<k$. By a suitable rescaling if necessary assume
$$
|h_{i_{k}}(\alpha)| > |\sum_{l\leq i<k} h_{i_{j}}(\alpha)|,
$$
for any $\alpha \in \ol\Delta_{i_{k}}$, $0 \leq k \leq M$ and $l < k$. Define $\bar{h} = \sum_{k=0}^{M} h_{i_{k}}$, then
$$
\bar{h}(\alpha) = \sum_{j\leq k} h_{i_{j}}(\alpha) = h_{i_{k}}(\alpha) + \sum_{j<k}h_{i_{j}}(\alpha).
$$
From this it follows that, for $\alpha \in \ol\Delta_{i_{k}}$, 
$$
|\bar{h}(\alpha)| = |h_{i_{k}}(\alpha) + \sum_{j<k}h_{i_{j}}(\alpha)| \geq ||h_{i_{k}}(\alpha)| - |\sum_{j<k}h_{i_{j}}(\alpha)|| >0.
$$
Thus $\bar{h}$ belongs to a Weyl chamber of the root system $\Delta$ and so defines a set of positive roots $\Delta_{+}$. Moreover, a root $\alpha \in \ol\Delta_{i_{k}}$ will be positive if and only if $h_{i_{k}}(\alpha) > 0$. \\

\subsection{} \label{sect: 3.2}
Let $\Lambda_{0} = \Lambda \cap \ol{\Delta}_{0}$, then $\Lambda_{0}$ defines a subsystem of roots and one obtains a parabolic $\mathfrak{q} \subset \lie$ from this system with a Levi factor defined by the fixed simple roots $\Lambda_{0}$. Let $\mathfrak{n}$ and $\mathfrak{l}$ denote the nilpotent radical and this Levi factor of $\mathfrak{q}$, respectively, and denote by $Q$ the subgroup with Lie algebra $\mathfrak{q}$ and by $N$ and $L$ the subgroups with Lie algebras $\mathfrak{n}$ and $\mathfrak{l}$, respectively. Let
$$
Z = \{ z \in L : szs^{-1} = z \},
$$
that is, $Z$ is the centraliser of $s$ in $L$ and let
$$
N_{s} = \{ n \in N : sns^{-1} \in N_{-} \}
$$
where $N_{-}$ is the opposite unipotent radical to $N$. Define $\Sigma_{s} = N_{s}Zs^{-1}$ and note that $Z$ normalizes both $N$ and $N_{s}$. Observe that the parabolic $Q$ obtained from $s$ is such that any conjugate of $s$ in $W$ will give rise to a parabolic conjugate to $Q$. Therefore, if $s, s' \in W$ lie in the same conjugacy class then the subvarieties $\Sigma_{s}$ and $\Sigma_{s'}$ are isomorphic, however, it is not yet clear whether the converse is also true in general. It is, however, important for the reader to note that, in general, the number of conjugacy classes in $W$ is greater than the number of nilpotent orbits in $\lie$. Therefore, for a given algebraic group the number of non-isomorphic deformed Poisson W-algebras is greater than the number of non-isomorphic W-algebras; see \cite{Sev}. \\
We refer the reader to \cite{Sev} for the proof(s) of the following result which states that $\Sigma_{s}$ is transverse to the set of conjugacy classes in $G$ and for which an analogue of the Kostant cross-section theorem holds.

\begin{theorem}[\cite{Sev}, Proposition 2.1 and Proposition 2.3] \label{prop:3.2a}
The conjugation map $G \times \Sigma_{s} \rightarrow G$ has a surjective differential and the restriction $N \times \Sigma_{s} \rightarrow NZs^{-1}N$ is an isomorphism of varieties.
\end{theorem}
\hfill

\section{A System of Positive Roots}
\subsection{}
Throughout this section $G = SL_{l+1}(\mathbb{C})$, the group of $(l+1)\times(l+1)$ invertible matrices with unit determinant, which means $\lie = \mathfrak{sl}_{l+1}$, the Lie algebra of traceless $(l+1)\times(l+1)$ matrices. The Weyl group of $G$ is $W = S_{l+1}$. For $1 \leq a,b \leq l+1$ let $E_{a,b}$ denote the $(l+1)\times(l+1)$ matrix which has 1 in the $(a,b)$-th position and 0 elsewhere. Unless otherwise stated, let $\Delta$ be the root system of the pair $(\lie,\lieh)$ where $\lieh$ is the Cartan subalgebra consisting of diagonal matrices in $\lie$. Identify $\lieh$ with the subspace $\bigoplus_{i=1}^{l}\mathbb{C}(e_{i}-e_{i+1}) \subset \mathbb{C}^{l+1}$ where $e_{i} \in \mathbb{C}^{l+1}$ is the vector with 1 in the $i$-th position and 0 elsewhere. Moreover, $\Delta_{+}$ is the system of positive roots in $\Delta$ such that the root subspace of $\alpha_{i}+\alpha_{i+1}+\cdots+\alpha_{j}$, for $1 \leq i \leq j \leq l$, in $\lie$ is generated by $E_{i,j+1}$. \\
Define the maps 
$$
\text{row}: \Delta \rightarrow \mathbb{Z}_{>0}; \ \text{row}(\alpha) = \left\{\begin{array}{ll} i, & \text{if $\alpha = \alpha_{i} + \cdots + \alpha_{j} \in \Delta_{+}$} \\
                                                                                                                          j+1, & \text{if $\alpha = -(\alpha_{i} + \cdots + \alpha_{j}) \in \Delta_{-}$}
\end{array}\right.
$$
and
$$
\text{col}: \Delta \rightarrow \mathbb{Z}_{>0}; \ \text{col}(\alpha) = \left\{\begin{array}{ll} j+1, & \text{if $\alpha = \alpha_{i} + \cdots + \alpha_{j} \in \Delta_{+}$} \\
                                                                                                                          i, & \text{if $\alpha = -(\alpha_{i} + \cdots + \alpha_{j}) \in \Delta_{-}$.}
\end{array}\right.
$$
For any $c \in \mathbb{C}$ and $\alpha \in \Delta$ define $X_{\alpha}(c) = \exp(cE_{\alpha})$, where
$$
E_{\alpha} =  \left\{ \begin{array}{ll} E_{i,j+1}, & \text{if $\alpha = \alpha_{i} + \cdots + \alpha_{j}$} \\
                                                             E_{j+1,i}, & \text{if $\alpha = -(\alpha_{i} + \cdots + \alpha_{j})$,}
\end{array}\right.
$$
so that $c \mapsto X_{\alpha}(c)$ is a one-parameter subgroup in $G$. We will say that $X_{\alpha}$ is the one-parameter subgroup associated to $\alpha$ or that $X_{\alpha}$ is an $\alpha$-one-parameter subgroup. Finally, define the floor and ceiling functions $\lfloor\cdot\rfloor:\mathbb{R} \rightarrow \mathbb{Z}$ and $\lceil\cdot\rceil:\mathbb{R} \rightarrow \mathbb{Z}$, respectively, by $\lfloor x \rfloor = \text{max}\{ y \in \mathbb{Z} : y \leq x \}$ and $\lceil x \rceil = \text{min}\{ y \in \mathbb{Z} : y \geq x \}$. \\

\subsection{}
The varieties $\Sigma_{s}$ depend on conjugacy classes in $W$ which in type A are cycle types in $S_{l+1}$. These cycle types are parametrized by partitions of $l+1$. Therefore, the number of non-isomorphic deformed Poisson W-algebras in type A is at most the number of partitions of $l+1$. Compare this with nilpotent classes in type A which are also parametrized by partitions of $l+1$. It was mentioned in \ref{sect: 3.2} that, in general, the number of non-isomorphic deformed Poisson W-algebras of $G$ is greater than the number of non-isomorphic W-algebras of $\lie$. It therefore follows that there are the same number of non-isomorphic W-algebras as non-isomorphic deformed Poisson W-algebras in type A and, moreover, that the varieties $\Sigma_{s}$ don't just depend on the conjugacy classes of $W$ but are in fact parametrized by the conjugacy classes of $W$. \\

\subsection{}
In \ref{sect: 3.1} a method was outlined for obtaining a system of positive roots from any Weyl group element and from this we constructed the subvariety $\Sigma_{s}$. For a general Weyl group element it is however difficult to deduce which roots will form a positive root system. Since these slices are parametrized by conjugacy classes in $W$ we would ideally like to find conjugacy class representatives that will give the associated positive root system $\Delta_{+}$ a predictable form. To this end, we state a result of Carter which proves to be useful.

\begin{theorem}[\cite{Car}, Theorem C]
Every element of a Weyl group is expressible as a product of two involutions.
\end{theorem}

In particular, any $s \in W$ can be expressed as the product of two involutions $s = s_{1}s_{2}$ where $s_{1} = s_{\gamma_{1}} \cdots s_{\gamma_{n}}$ and $s_{2} = s_{\gamma_{n+1}} \cdots s_{\gamma_{l'}}$ such that the roots in each of the sets $\{ \gamma_{1}, \ldots, \gamma_{n} \}$ and $\{ \gamma_{n+1}, \ldots, \gamma_{l'} \}$ are positive and mutually orthogonal. \\
The method of associating a system of positive roots to a Weyl group element can just as easily be reversed, by which we mean, for any given positive root system $\Delta_{+}$ we may ask which element or elements in $W$ are associated to $\Delta_{+}$ in the sense of \ref{sect: 3.1}. \\
As conjugacy classes of $W$ are parametrized by their cycle type, which are products of disjoint cycles, from this point onward we will assume $s$ is a single cycle of length greater than one. In the following proposition, for any given system of positive roots, we state a representative $s \in W$ of the conjugacy class of single cycles associated to this system. The form of this representative splits into four cases.

\begin{proposition} \label{prop:4.1a}
Let $s \in W$ be a single cycle such that $s$ is represented as a product, $s=s_{1}s_{2}$, of two involutions $s_{1} = s_{\gamma_{1}} \cdots s_{\gamma_{n}}$ and $s_{2} = s_{\gamma_{n+1}} \cdots s_{\gamma_{l'}}$ and let $\Delta_{+}$ be any system of positive roots. The root system $\Delta_{+}$ is associated to $s$ when the sets $\{\gamma_{1}, \ldots, \gamma_{n}\}$, $\{\gamma_{n+1}, \ldots, \gamma_{l'}\}$ have one of the following forms:
\begin{itemize}
\item[(i)] For $m$ even and $l'$ odd,
\begin{align*}
\{\gamma_{1}, \ldots, \gamma_{n}\} &= \{\alpha_{2}, \alpha_{4}, \cdots, \alpha_{m}, \alpha_{m+p+2}, \alpha_{m+p+4}, \cdots, \alpha_{m+p+m}\} \\
\{\gamma_{n+1}, \ldots, \gamma_{l'}\} &= \{\alpha_{1}, \alpha_{3}, \cdots, \alpha_{m-1}, \gamma, \alpha_{m+p+3}, \alpha_{m+p+5}, \cdots, \alpha_{m+p+m+1}\}
\end{align*}

\item[(ii)] For $m$ even and $l'$ even,
\begin{align*}
\{\gamma_{1}, \ldots, \gamma_{n}\} &= \{\alpha_{2}, \alpha_{4}, \cdots, \alpha_{m}, \alpha_{m+p+2}, \alpha_{m+p+4}, \cdots, \alpha_{m+p+m+2}\}\\
\{\gamma_{n+1}, \ldots, \gamma_{l'}\} &= \{\alpha_{1}, \alpha_{3}, \cdots, \alpha_{m-1}, \gamma, \alpha_{m+p+3}, \alpha_{m+p+5}, \cdots, \alpha_{m+p+m+1}\}
\end{align*}

\item[(iii)] For $m$ odd and $l'$ odd,
\begin{align*}
\{\gamma_{1}, \ldots, \gamma_{n}\} &= \{\alpha_{1}, \alpha_{3}, \cdots, \alpha_{m}, \alpha_{m+p+2}, \alpha_{m+p+4}, \cdots, \alpha_{m+p+m+1}\} \\
\{\gamma_{n+1}, \ldots, \gamma_{l'}\} &= \{\alpha_{2}, \alpha_{4}, \cdots, \alpha_{m-1}, \gamma, \alpha_{m+p+3}, \alpha_{m+p+5}, \cdots, \alpha_{m+p+m}\}
\end{align*}

\item[(iv)] For $m$ odd and $l'$ even,
\begin{align*}
\{\gamma_{1}, \ldots, \gamma_{n}\} &= \{\alpha_{1}, \alpha_{3}, \cdots, \alpha_{m}, \alpha_{m+p+2}, \alpha_{m+p+4}, \cdots, \alpha_{m+p+m+1}\} \\
\{\gamma_{n+1}, \ldots, \gamma_{l'}\} &= \{\alpha_{2}, \alpha_{4}, \cdots, \alpha_{m-1}, \gamma, \alpha_{m+p+3}, \alpha_{m+p+5}, \cdots, \alpha_{m+p+m+2}\}
\end{align*}
\end{itemize}
where $m=\lfloor\frac{1}{2}(l'-1)\rfloor$, $p=l-l'$ and $\gamma = \alpha_{m+1} + \cdots + \alpha_{m+p+1}$.
\end{proposition}

\begin{proof}
We consider case (i) of the proposition, the proof of the other cases being similar. Assume that $s = s_{1}s_{2}$ where $s_{1} = s_{\gamma_{1}} \cdots s_{\gamma_{n}}$, $s_{2} = s_{\gamma_{n+1}} \cdots s_{\gamma_{l'}}$ and the roots $\gamma_{i}$, $1 \leq i \leq l'$, are given in (i). This is indeed a single cycle in $W$ and has order $l'+1$. We show that $\Delta_{+}$ is associated to $s$, first by showing that all roots not fixed by the action of $s$ have non-zero orthgonal projection onto $\lieh_{K}$, i.e. $\Delta = \ol{\Delta}_{0}\cup\ol{\Delta}_{K}$, then by showing that all projections of roots from $(\ol{\Delta}_{K})_{+} = \ol{\Delta}_{K}\cap\Delta_{+}$ onto $\lieh_{K}$ lie in a half plane. \\
The element $s$ acts as an orthogonal transformation on $\lieh_{\mathbb{R}}$ with inner product $(\cdot, \cdot)_{\lieh_{\mathbb{R}}}$, the restriction of the Killing form on $\lieh$ to $\lieh_{\mathbb{R}}$. Extend this to a sesquilinear product $(\cdot, \cdot)_{\lieh}$ on $\lieh$ by defining
$$
(v,v')_{\lieh} = (u,u')_{\lieh_{\mathbb{R}}} + (w,w')_{\lieh_{\mathbb{R}}} + i(w,u')_{\lieh_{\mathbb{R}}} - i(u,w')_{\lieh_{\mathbb{R}}} 
$$
for $v,v' \in \lieh$ such that $v= u+iw$, $v'= u'+iw'$ with $u,u',w,w' \in \lieh_{\mathbb{R}}$ and $(u,u')_{\lieh} = (u,u')_{\lieh_{\mathbb{R}}}$. With respect to this sesquilinear product $s$ may be extended to a unitary operator on $\lieh$. If $v = (v_{1}, \ldots, v_{l+1}) \in \lieh$ where, for $1 \leq i \leq l+1$, $v_{i}$ are the components of $v$ with respect to the basis $e_{1}, \ldots, e_{l+1}$, the action of $s$ on $\lieh$ is given by 
$$
s(v_{1}, \ldots, v_{l+1}) = (v_{s^{-1}(1)}, \ldots, v_{s^{-1}(l+1)})
$$
where $s^{-1}(i)$ is the action of $s^{-1}$ on $i$ induced by the permutation group action of $S_{l+1}$ on the set $\{1,2,\ldots,l+1\}$. This unitary operator has order $l'+1$ thus if $\lambda$ is an eigenvalue of $s$ and $v^{\lambda} \in \lieh$ is a non-zero $\lambda$-eigenvector then
$$
v^{\lambda} = s^{l'+1}v^{\lambda} = \lambda^{l'+1}v^{\lambda}
$$
and hence
$$
\lambda^{l'+1} = 1,
$$
that is, the eigenvalues of $s$ are $(l'+1)$-th roots of unity. This gives the spectral decomposition
$$
\lieh = \bigoplus_{\lambda \in \mathbb{C}:\lambda^{l'+1}=1} \lieh(\lambda)
$$
of $\lieh$ where $\lieh(\lambda)$ is the $\lambda$-eigenspace of $s$ and is spanned by the vector $v^{\lambda}=(v_{1}^{\lambda},\ldots, v_{l+1}^{\lambda})$ where
\begin{align*}
v_{i}^{\lambda} &= \left\{ \begin{array}{ll} \lambda^{\frac{1}{2}(4m-i+5)}, & \text{if $1 \leq i \leq m+1$ and $i$ odd} \\
                                                   \lambda^{\frac{1}{2}i}, & \text{if $1 \leq i \leq m+1$ and $i$ even} \\
                                                   0, & \text{if $m+2 \leq i \leq m+p+1$},
\end{array}\right. \\
v_{m+p+i}^{\lambda} &= \left\{ \begin{array}{ll}  \lambda^{\frac{1}{2}(m+i)}, & \text{if $2 \leq i \leq m+2$ and $i$ even} \\
                                                              \lambda^{\frac{1}{2}(3m-i+5)}, & \text{if $2 \leq i \leq m+2$ and $i$ odd}
\end{array}\right.
\end{align*}
for $\lambda \neq 1$. When $\lambda = 1$ the subspace $\lieh(1)$ is spanned by the vectors $\alpha_{m+2}, \ldots, \alpha_{m+p}$. \\
Write $v^{\lambda}=u^{\lambda}+iw^{\lambda}$, $u^{\lambda},w^{\lambda} \in \lieh_{\mathbb{R}}$, and note that $\ol{v}^{\lambda} = u^{\lambda} - iw^{\lambda}$ is a $\ol{\lambda}$-eigenvector as $s$ commutes with complex conjugation. Moreover,
$$
(v^{\lambda},\ol{v}^{\lambda})_{\lieh} = (\lambda^{-1}sv^{\lambda}, \ol{\lambda}^{-1}s\ol{v}^{\lambda})_{\lieh} = \lambda^{-2}(sv^{\lambda},s\ol{v}^{\lambda})_{\lieh} = \lambda^{-2}(v^{\lambda},\ol{v}^{\lambda})_{\lieh}.
$$
It follows that $(v^{\lambda},\ol{v}^{\lambda})_{\lieh} = 0$ or $\lambda^{2} = 1$. For $\lambda^{2} \neq 1$,
$$
0 = (v^{\lambda},\ol{v}^{\lambda})_{\lieh} = (u^{\lambda}+iw^{\lambda}, u^{\lambda}-iw^{\lambda})_{\lieh} = (u^{\lambda},u^{\lambda})_{\lieh}-(w^{\lambda},w^{\lambda})_{\lieh}+2i(u^{\lambda},w^{\lambda})_{\lieh}.
$$
Comparing real and imaginary parts gives $(u^{\lambda},u^{\lambda})_{\lieh}=(w^{\lambda},w^{\lambda})_{\lieh}$ and $(u^{\lambda},w^{\lambda})_{\lieh}=0$. Write $\lambda = e^{\frac{2\pi ik}{l'+1}}$ for $1 \leq k \leq l'$ and define $\lieh^{\lambda} = \langle u^{\lambda},w^{\lambda} \rangle_{\mathbb{R}}$, the real plane in $\lieh_{\mathbb{R}}$ spanned by the orthogonal vectors $u^{\lambda},w^{\lambda}$. It follows from 
\begin{align*}
sv^{\lambda}&=\lambda v^{\lambda}, \\
su^{\lambda} + isw^{\lambda} &= \big(\cos(\frac{2\pi k}{l'+1}) + i\sin(\frac{2\pi k}{l'+1})\big)(u^{\lambda}+iw^{\lambda}), \\
su^{\lambda} + isw^{\lambda} &= (\cos(\frac{2\pi k}{l'+1})u^{\lambda} - \sin(\frac{2\pi k}{l'+1})w^{\lambda}) + i(\sin(\frac{2\pi k}{l'+1})u^{\lambda}+\cos(\frac{2\pi k}{l'+1})w^{\lambda})
\end{align*}
that $s$ acts on $\lieh^{\lambda}$ in $\lieh_{\mathbb{R}}$ by rotation by $\frac{2\pi k}{l'+1}$ and that $\lieh^{\lambda} = \lieh_{i}$ for some $1 \leq i \leq K$. Fix $\lambda = e^{\frac{2\pi i}{l'+1}}$ so that $\lieh^{\lambda} = \lieh_{K}$.\\
It is clear that when $s$ is of the form (i) the simple roots fixed by the action of $s$ are $\alpha_{m+2},\ldots, \alpha_{m+p}$. We now show that the elements of $\Delta\backslash\ol{\Delta}_{0}$ have non-zero projection onto the plane $\lieh_{K}$ in $\lieh_{\mathbb{R}}$. Note that it suffices to verify this for the simple roots in $\Delta\backslash\ol{\Delta}_{0}$ only as $\ol{\Delta}_{0}$ is the root system of the Levi factor $\mathfrak{l}$ and every root from $\Delta \backslash \ol{\Delta}_{0}$ has at least one simple root from $\Delta \backslash \ol{\Delta}_{0}$ in its decomposition with respect to the system of simple roots. Let $P: \lieh_{\mathbb{R}} \rightarrow \lieh_{K}$ denote the orthogonal projection operator. For $\alpha \in \lieh_{\mathbb{R}}$
$$
P\alpha = \frac{(\alpha, u^{\lambda})_{\lieh}}{(u^{\lambda},u^{\lambda})_{\lieh}^{2}}u^{\lambda} + \frac{(\alpha,w^{\lambda})_{\lieh}}{(w^{\lambda},w^{\lambda})_{\lieh}^{2}}w^{\lambda} = \frac{1}{(u,^{\lambda}u^{\lambda})_{\lieh}}[(\alpha,u^{\lambda})_{\lieh}\hat{u}^{\lambda} + (\alpha,w^{\lambda})_{\lieh}\hat{w}^{\lambda}]
$$
where $\hat{u}^{\lambda}=\frac{u^{\lambda}}{(u^{\lambda},u^{\lambda})_{\lieh}}, \hat{w}^{\lambda}=\frac{w^{\lambda}}{(w^{\lambda},w^{\lambda})_{\lieh}}$. We may compute this orthogonal projection using the complex vector space $\lieh$ in the following way. For $\alpha \in \lieh_{\mathbb{R}} \subset \lieh$,
$$
(\alpha,v^{\lambda})_{\lieh} = (\alpha, u^{\lambda}+iw^{\lambda})_{\lieh} = (\alpha,u^{\lambda})_{\lieh} + i(\alpha,w^{\lambda})_{\lieh}.
$$
Define an isomorphism of real vector spaces $f:\lieh_{K} \rightarrow \mathbb{C}$ which maps the basis vectors $\hat{u}^{\lambda}, \hat{v}^{\lambda}$ as $f(\hat{u}^{\lambda}) = 1$ and $f(\hat{w}^{\lambda})=i$. Thus $P\alpha$ is mapped to $\frac{1}{(u^{\lambda},u^{\lambda})_{\lieh}}(\alpha,v^{\lambda})_{\lieh}$ by $f$ and the vectors $P\alpha$, for $\alpha \in \Delta\backslash\ol{\Delta}_{0}$, are non-zero and lie in a half plane in $\lieh_{\mathbb{R}}$ if and only if the complex numbers $(\alpha,v^{\lambda})_{\lieh}$, $\alpha \in \Delta\backslash\ol{\Delta}_{0}$, are non-zero and lie in a half plane in $\mathbb{C}$. When $\lambda^{2}=1$ the element $s$ is a single transposition, $s=s_{\gamma}$, and $\lieh$ decomposes as $\lieh = \lieh(1) \oplus \lieh(-1)$. Moreover, $\lieh(-1)$ is one-dimensional and spanned by $u \in \lieh_{\mathbb{R}}$ as the action of $s$ on $\lieh$ induces a reflection on $\lieh_{\mathbb{R}}$. It follows that $\lieh_{\mathbb{R}} = \lieh_{0} \oplus \lieh_{K}$ where $\lieh_{K}$ is the one-dimensional line in $\lieh_{\mathbb{R}}$ spanned by $u$. For $\alpha \in \lieh_{\mathbb{R}}$ the orthogonal projection of $\alpha$ onto the line $\lieh_{K}$ is 
$$
P\alpha = \frac{(\alpha,u^{\lambda})_{\lieh_{\mathbb{R}}}}{(u^{\lambda},u^{\lambda})_{\lieh_{\mathbb{R}}}}\hat{u}^{\lambda} = \frac{(\alpha,u^{\lambda})_{\lieh}}{(u^{\lambda},u^{\lambda})_{\lieh}}\hat{u}^{\lambda}
$$
which is equal to the orthogonal projection of $u^{\lambda}$ onto $\lieh(-1)$ in $\lieh$.\\
We now use the above method to verify that all elements of $\Delta\backslash\ol{\Delta}_{0}$ have non-zero projection onto the plane $\lieh_{K}$. Recall that it suffices to check this for the simple roots in $\Delta\backslash\ol{\Delta}_{0}$, i.e. for the roots $\alpha_{1}, \ldots, \alpha_{m+1}, \alpha_{m+p+1}, \ldots, \alpha_{m+p+m+1}$. Note that when $s$ is of the form given by (i) it follows that $l=2m+p+1$ which implies $\lambda = e^{\frac{2\pi i}{l'+1}} = e^{\frac{\pi i}{m+1}}$. It is easy to check that
\begin{align*}
(\alpha_{k}, v^{\lambda})_{\lieh} &= \left\{\begin{array}{ll}
2\sin\Big[\frac{(2m-k+2)\pi}{2(m+1)}\Big]e^{\frac{(3m+4)\pi i}{2(m+1)}}, & \text{if $1 \leq k \leq m$ and $k$ odd} \\
2\sin\Big[\frac{(2m-k+2)\pi}{2(m+1)}\Big]e^{\frac{\pi i}{2}}, & \text{if $1 \leq k \leq m$ and $k$ even} \\
e^{\frac{(3m+4)\pi i}{2(m+1)}}, & \text{if $k=m+1$},
\end{array}\right. \\
(\alpha_{m+p+k}, v^{\lambda})_{\lieh} &= \left\{\begin{array}{ll}
e^{\frac{(3m+4)\pi i}{2(m+1)}}, & \text{if $k=1$} \\
2\sin\Big[\frac{(m-k+2)\pi}{2(m+1)}\Big]e^{\frac{\pi i}{2}}, & \text{if $2 \leq k \leq m+1$ and $k$ even} \\
2\sin\Big[\frac{(m-k+4)\pi}{2(m+1)}\Big]e^{\frac{3(m+2)\pi i}{2(m+1)}}, & \text{if $2 \leq k \leq m+1$ and $k$ odd.}
\end{array}\right.
\end{align*}
When $m$ is non-zero 
$$
-\frac{\pi}{2} < \arg(\alpha_{k},v^{\lambda})_{\lieh} \leq \frac{\pi}{2}
$$ 
for all $\alpha_{k} \in \{\alpha_{1}, \ldots, \alpha_{m+1}, \alpha_{m+p+1}, \ldots, \alpha_{m+p+m+1}\}$ and thus all such $\alpha_{k}$ lie in a half-plane in $\mathbb{C}$. When $m$ is zero $(\alpha_{1}, v^{\lambda})_{\lieh} = (\alpha_{p+1}, v^{\lambda})_{\lieh} = 1$ and thus $P\alpha_{1}, P\alpha_{p+1}$ lie in the same half-line. The stated conjugacy class representative $s \in W$ therefore defines the given system of positive roots.
\end{proof}
\hfill

\subsection{} \label{sect: 4.4}
From this point onward an element of $W$ denoted by $s$ will be a product of two involutions $s=s_{1}s_{2}$ such that $s_{1} = s_{\gamma_{1}} \cdots s_{\gamma_{n}}$, $s_{2} = s_{\gamma_{n+1}} \cdots s_{\gamma_{l'}}$ and the sets $\{\gamma_{1}, \ldots, \gamma_{n}\}$, $\{\gamma_{n+1}, \ldots, \gamma_{l'}\}$ are given by one of the forms in Proposition \ref{prop:4.1a}. In particular, $s$ is a single cycle and a conjugacy class representative for the conjugacy class in $W$ of single cycles. Define $\Delta_{s}$ and $\Delta_{s^{-1}}$ to be the set of positive roots in $\Delta$ that become negative roots under the action of $s$ and $s^{-1}$, respectively. It was shown in \cite{Sev1} that $\Delta_{s} = s_{2}\Delta_{s_{1}} \cup \Delta_{s_{2}}$ and $\Delta_{s^{-1}} = \Delta_{s_{1}} \cup s_{1}\Delta_{s_{2}}$ where $\Delta_{s_{1}}$ and $\Delta_{s_{2}}$ are the positive roots in $\Delta$ that become negative roots under the action of $s_{1}$ and $s_{2}$, respectively. \\
It will prove very useful later to explicitly describe the set $\Delta_{s^{-1}}$ in each case of the above proposition. If $s$ of the form (i) in Proposition \ref{prop:4.1a} then 
\begin{align*}
\Delta_{s^{-1}} = \{ &\alpha_{1}+\alpha_{2}, \alpha_{2}+\alpha_{3}+\alpha_{4},\ldots, \alpha_{m-2}+\alpha_{m-1}+\alpha_{m}, \alpha_{m}+\alpha_{m+1}, \alpha_{m}+\alpha_{m+1}+\alpha_{m+2}, \\
&\ldots, \alpha_{m}+\cdots+\alpha_{m+p}, \alpha_{m}+\cdots+\alpha_{m+p+2}, \alpha_{m+2}+\cdots+\alpha_{m+p+2}, \\
&\alpha_{m+3}+\cdots+\alpha_{m+p+2},\ldots, \alpha_{m+p+2}, \alpha_{m+p+2}+\alpha_{m+p+3}+\alpha_{m+p+4}, \\
&\alpha_{m+p+4}+\alpha_{m+p+5}+\alpha_{m+p+6}, \ldots, \alpha_{m+p+m-2}+\alpha_{m+p+m-1}+\alpha_{m+p+m}, \\
&\alpha_{m+p+m}+\alpha_{m+p+m+1}, \alpha_{2}, \alpha_{4}, \ldots, \alpha_{m}, \alpha_{m+p+2}, \alpha_{m+p+4}, \ldots, \alpha_{m+p+m} \},
\end{align*}
if $s$ of the form (ii) in Proposition \ref{prop:4.1a} then 
\begin{align*}
\Delta_{s^{-1}} = \{ &\alpha_{1}+\alpha_{2}, \alpha_{2}+\alpha_{3}+\alpha_{4}, \ldots, \alpha_{m-2}+\alpha_{m-1}+\alpha_{m}, \alpha_{m}+\alpha_{m+1}, \alpha_{m}+\alpha_{m+1}+\alpha_{m+2}, \\
&\ldots, \alpha_{m}+\cdots+\alpha_{m+p}, \alpha_{m}+\cdots+\alpha_{m+p+2}, \alpha_{m+2}+\cdots+\alpha_{m+p+2}, \\ &\alpha_{m+3}+\cdots+\alpha_{m+p+2}, \ldots, \alpha_{m+p+2}, \alpha_{m+p+2}+\alpha_{m+p+3}+\alpha_{m+p+4}, \\
&\alpha_{m+p+4}+\alpha_{m+p+5}+\alpha_{m+p+6}, \ldots, \alpha_{m+p+m}+\alpha_{m+p+m+1}+\alpha_{m+p+m+2}, \\
&\alpha_{2}, \alpha_{4}, \ldots, \alpha_{m}, \alpha_{m+p+2}, \alpha_{m+p+4}, \ldots, \alpha_{m+p+m+2} \},
\end{align*}
if $s$ of the form (iii) in Proposition \ref{prop:4.1a} then 
\begin{align*}
\Delta_{s^{-1}} = \{ &\alpha_{1}+\alpha_{2}+\alpha_{3}, \alpha_{3}+\alpha_{4}+\alpha_{5}, \ldots, \alpha_{m-2}+\alpha_{m-1}+\alpha_{m}, \alpha_{m}+\alpha_{m+1}, \\
&\alpha_{m}+\alpha_{m+1}+\alpha_{m+2}, \ldots, \alpha_{m}+\cdots+\alpha_{m+p}, \alpha_{m}+\cdots+\alpha_{m+p+2}, \\ &\alpha_{m+2}+\cdots+\alpha_{m+p+2}, \alpha_{m+3}+\cdots+\alpha_{m+p+2}, \ldots, \alpha_{m+p+2}, \\ &\alpha_{m+p+2}+\alpha_{m+p+3}+\alpha_{m+p+4}, \alpha_{m+p+4}+\alpha_{m+p+5}+\alpha_{m+p+6}, \\
&\ldots, \alpha_{m+p+m-1}+\alpha_{m+p+m}+\alpha_{m+p+m+1}, \alpha_{1}, \alpha_{3}, \ldots, \alpha_{m}, \alpha_{m+p+2}, \\
&\alpha_{m+p+4}, \ldots, \alpha_{m+p+m+1} \},
\end{align*}
if $s$ of the form (iv) in Proposition \ref{prop:4.1a} then 
\begin{align*}
\Delta_{s^{-1}} = \{ &\alpha_{1}+\alpha_{2}+\alpha_{3}, \alpha_{3}+\alpha_{4}+\alpha_{5}, \ldots, \alpha_{m-2}+\alpha_{m-1}+\alpha_{m}, \alpha_{m}+\alpha_{m+1}, \\
&\alpha_{m}+\alpha_{m+1}+\alpha_{m+2}, \ldots, \alpha_{m}+\cdots+\alpha_{m+p}, \alpha_{m}+\cdots+\alpha_{m+p+2}, \\ &\alpha_{m+2}+\cdots+\alpha_{m+p+2}, \alpha_{m+3}+\cdots+\alpha_{m+p+2}, \ldots, \alpha_{m+p+2}, \\ &\alpha_{m+p+2}+\alpha_{m+p+3}+\alpha_{m+p+4}, \alpha_{m+p+4}+\alpha_{m+p+5}+\alpha_{m+p+6}, \\
&\ldots, \alpha_{m+p+m-1}+\alpha_{m+p+m}+\alpha_{m+p+m+1}, \alpha_{m+p+m+1}+\alpha_{m+p+m+2}\\
&\alpha_{1}, \alpha_{3}, \ldots, \alpha_{m}, \alpha_{m+p+2}, \alpha_{m+p+4}, \ldots, \alpha_{m+p+m+1} \}. 
\end{align*}
\hfill

\subsection{}
Define $\gamma' = \alpha_{m} + \cdots + \alpha_{m+p+2}$ and note that $\gamma' \in \Delta_{s^{-1}}$ for all $s$. Moreover, it is important to observe that each root in $\Delta_{s^{-1}}$ is at most the sum of two roots in $\Delta_{s^{-1}}$ and in fact the only root in $\Delta_{s^{-1}}$ that is the sum of two such roots is $\gamma'$. If we denote by $\mathcal{O}_{\gamma'}$ the orbit of $\gamma'$  in $\Delta$ with respect to the action of the group generated by $s$ then
\begin{align*}
\mathcal{O}_{\gamma'}\cap(\ol{\Delta}_{K})_{+} = \{ &\gamma', \alpha_{m-2}+\cdots+\alpha_{m+p+4}, \alpha_{m-4}+\cdots+\alpha_{m+p+6}, \ldots, \alpha_{2}+\cdots+\alpha_{m+p+m}, \\
&\alpha_{1}+\cdots+\alpha_{m+p+m+1}, \alpha_{3}+\cdots+\alpha_{m+p+m-1}, \alpha_{5}+\cdots+\alpha_{m+p+m-3}, \\
&\ldots, \alpha_{m-1}+\cdots+\alpha_{m+p+3}, \gamma \}
\end{align*}
when $s$ is of the form (i) in Proposition \ref{prop:4.1a},
\begin{align*}
\mathcal{O}_{\gamma'}\cap(\ol{\Delta}_{K})_{+} = \{ &\gamma', \alpha_{m-2}+\cdots+\alpha_{m+p+4}, \alpha_{m-4}+\cdots+\alpha_{m+p+6}, \ldots, \alpha_{2}+\cdots+\alpha_{m+p+m}, \\
&\alpha_{1}+\cdots+\alpha_{m+p+m+2}, \alpha_{3}+\cdots+\alpha_{m+p+m+1}, \alpha_{5}+\cdots+\alpha_{m+p+m-1}, \\
&\ldots, \alpha_{m-1}+\cdots+\alpha_{m+p+5}, \alpha_{m+1}+\cdots+\alpha_{m+p+3} \}
\end{align*}
when $s$ is of the form (ii) in Proposition \ref{prop:4.1a},
\begin{align*}
\mathcal{O}_{\gamma'}\cap(\ol{\Delta}_{K})_{+} = \{ &\gamma', \alpha_{m-2}+\cdots+\alpha_{m+p+4}, \alpha_{m-4}+\cdots+\alpha_{m+p+6}, \\
&\ldots, \alpha_{1}+\cdots+\alpha_{m+p+m+1}, \alpha_{2}+\cdots+\alpha_{m+p+m}, \alpha_{4}+\cdots+\alpha_{m+p+m-2}, \\ &\alpha_{6}+\cdots+\alpha_{m+p+m-4}, \ldots, \alpha_{m-1}+\cdots+\alpha_{m+p+3}, \gamma \}
\end{align*}
when $s$ is of the form (iii) in Proposition \ref{prop:4.1a} and
\begin{align*}
\mathcal{O}_{\gamma'}\cap(\ol{\Delta}_{K})_{+} = \{ &\gamma', \alpha_{m-2}+\cdots+\alpha_{m+p+4}, \alpha_{m-4}+\cdots+\alpha_{m+p+6}, \\
&\ldots, \alpha_{1}+\cdots+\alpha_{m+p+m+1}, \alpha_{2}+\cdots+\alpha_{m+p+m+2}, \alpha_{4}+\cdots+\alpha_{m+p+m}, \\ &\alpha_{6}+\cdots+\alpha_{m+p+m-2}, \ldots, \alpha_{m-1}+\cdots+\alpha_{m+p+5}, \\
&\alpha_{m+1}+\cdots+\alpha_{m+p+3} \}
\end{align*}
when $s$ is of the form (iv) in Proposition \ref{prop:4.1a}. The description of these sets will become useful in a moment. \\

\begin{figure} \label{fig:1}
\centering
\begin{tikzpicture}
\draw (-5,0) -- (5,0);
\draw (0:0) -- (20:5) (0:0) -- (40:5) (0:0) -- (60:5);
\draw[fill] (80:2.5) circle [radius=0.01];
\draw[fill] (90:2.5) circle [radius=0.01];
\draw[fill] (100:2.5) circle [radius=0.01];
\draw (0:0) -- (130:5) (0:0) -- (150:5) (0:0) -- (170:5);
\draw (1,0) arc (0:20:1);
\node at (10:5) {$\ol{\Delta}_{K}^{1}$};
\node at (30:5) {$\ol{\Delta}_{K}^{2}$};
\node at (50:5) {$\ol{\Delta}_{K}^{3}$};
\node at (140:5) {$\ol{\Delta}_{K}^{D-1}$};
\node at (160:5) {$\ol{\Delta}_{K}^{D}$};
\node at (175:5) {$\ol{\Delta}_{K}^{D+1}$};
\node at (10:1.3) {$\theta_{K}$};
\end{tikzpicture}
\caption{Sectors of the half-plane in $\lieh_{K}$ onto which the roots from the sets $\ol{\Delta}_{K}^{1}, \ldots, \ol{\Delta}_{K}^{D+1}$ are orthogonally projected.}
\end{figure}
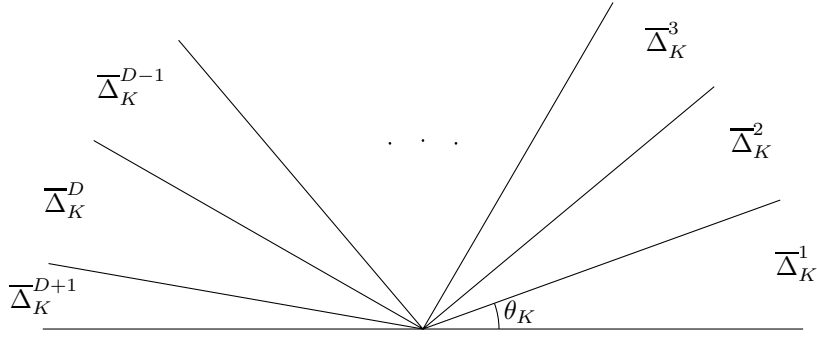

\subsection{}
As has already been described one can use the Killing form to identify $\lieh$ with $\lieh^{*}$ and orthogonally project all elements of $\ol{\Delta}_{K}$ onto the plane $\mathfrak{h}_{K}$. Now $\alpha \in (\ol{\Delta}_{K})_{+}$ if and only if $h_{K}(\alpha) > 0$ and all projected elements of $(\ol{\Delta}_{K})_{+}$ lie in a half-plane in $\mathfrak{h}_{K}$. Using the fact that $s$ acts on $\mathfrak{h}_{K}$ by rotation one can express $(\ol{\Delta}_{K})_{+}$ as the disjoint union
\begin{equation} \label{eq:4.1a}
(\ol{\Delta}_{K})_{+} = \bigcup_{k = 1}^{D+1} \ol{\Delta}_{K}^{k}
\end{equation}
where $\ol{\Delta}_{K}^{k} = \{ \alpha \in (\ol{\Delta}_{K})_{+} : s^{-1}\alpha, \ldots, s^{-(k-1)}\alpha \in (\ol{\Delta}_{K})_{+}, s^{-k}\alpha \in (\ol{\Delta}_{K})_{-} \}$ for each $1 \leq k \leq D+1$; see Figure 1. Let $d : (\ol{\Delta}_{K})_{+} \rightarrow \mathbb{Z}$ denote the function defined by the restriction $d\vert_{\ol{\Delta}_{K}^{k}} = k$. It will be of use to consider another partition of $(\ol{\Delta}_{K})_{+}$ given by
\begin{equation} \label{eq:4.1a1}
(\ol{\Delta}_{K})_{+} = \Delta^{C}\cup\Delta^{R}\cup\Delta^{O}
\end{equation}
where $\Delta^{C} = \{ \alpha \in (\ol{\Delta}_{K})_{+} : \text{row}(s\alpha) = \text{row}(\alpha) \}$, $\Delta^{R} = \{ \alpha \in (\ol{\Delta}_{K})_{+} : \text{col}(s\alpha) = \text{col}(\alpha) \}$ and $\Delta^{O} = (\ol{\Delta}_{K})_{+}\backslash(\Delta^{C}\cup\Delta^{R})$ and a partition of $\Delta^{O}$ given by
\begin{equation*}
\Delta^{O} = \Delta_{1}^{O}\cup\Delta_{2}^{O}\cup\Delta_{3}^{O}
\end{equation*}
where $\Delta_{1}^{O} = \{ \alpha \in \Delta^{O} : \text{row}(\alpha) \geq m+p+2 \}$, $\Delta_{2}^{O} = \{ \alpha \in \Delta^{O} : \text{col}(\alpha) \leq m+1 \}$ and $\Delta_{3}^{O} = \Delta^{O} \backslash (\Delta_{1}^{O}\cup\Delta_{2}^{O})$. Finally, denote by $\Delta_{Z}$ the root system of $Z$ and partition $\Delta_{Z}$ by
\begin{equation} \label{eq:4.1b}
\Delta_{Z} = \Delta_{Z}^{m+2}\cup\cdots\cup\Delta_{Z}^{m+p+1}
\end{equation}
where $\Delta_{Z}^{i} = \{ \alpha \in \Delta_{Z} : \text{col}(\alpha) = i \}$, for each $m+2 \leq i \leq m+p+1$. \\

\subsection{}
With the above subsets of $(\ol{\Delta}_{K})_{+}$ now defined, before we proceed, we observe that $\mathcal{O}_{\gamma'}\cap(\ol{\Delta}_{K})_{+} \subset \Delta_{3}^{O}$ and if $\eta_{1}, \eta_{2} \in \Delta_{s^{-1}}$ such that $\gamma' = \eta_{1}+\eta_{2}$ then, without loss of generality, $\eta_{1} \in \Delta_{s^{-1}}\cap\Delta^{C}$ and $\eta_{2} \in \Delta_{s^{-1}}\cap\Delta^{R}$. We now prove some basic properties concerning the roots in $(\ol{\Delta}_{K})_{+}\cup\Delta_{Z}$.

\begin{lemma} \label{lemma:4.1a}
Let $1 \leq k \leq D+1$, $\alpha \in \Delta_{Z}$, $\beta_{1},\beta_{2} \in \Delta^{C}$, $\delta_{1}, \delta_{2} \in \Delta^{R}$, $\eta \in \Delta^{O}$, $\beta \in \ol{\Delta}_{K}^{k} \cap \Delta^{C}$ and $ \delta \in \ol{\Delta}_{K}^{k} \cap \Delta^{R}$. One has
\begin{itemize}
\item[(i)] $d(\beta_{1}) = d(\beta_{2})$ if and only if $\text{col}(\beta_{1}) = \text{col}(\beta_{2})$,
\item[(ii)] $d(\delta_{1}) = d(\delta_{2})$ if and only if $\text{row}(\delta_{1}) = \text{row}(\delta_{2})$,
\item[(iii)] $\alpha + \beta \in \Delta$ implies $\alpha + \beta \in \ol{\Delta}_{K}^{k} \cap \Delta^{C}$,
\item[(iv)] $\alpha + \delta \in \Delta$ implies $\alpha + \delta \in \ol{\Delta}_{K}^{k} \cap \Delta^{R}$,
\item[(v)] $\alpha + \eta \notin \Delta$.
\end{itemize}
\end{lemma}

\begin{proof}
For (i), suppose first that $d(\beta_{1}) = d(\beta_{2}) = k$. Recall the set $\Delta_{s^{-1}}$ whose elements are the positive roots that become negative roots under the action of $s^{-1}$ and observe from the descriptions of $\Delta_{s^{-1}}$ given in \ref{sect: 4.4} that $\text{col}(\eta_{1}) = \text{col}(\eta_{2})$ for all $\eta_{1},\eta_{2} \in \Delta_{s^{-1}} \cap \Delta^{C}$. It follows from our supposition that $s^{-(k-1)}\beta_{1}, s^{-(k-1)}\beta_{2} \in \Delta_{s^{-1}}\cap\Delta^{C}$ which implies $\text{col}(s^{-(k-1)}\beta_{1}) = \text{col}(s^{-(k-1)}\beta_{2})$ and thus $\text{col}(\beta_{1}) = \text{col}(\beta_{2})$. Conversely, if $\text{col}(\beta_{1}) = \text{col}(\beta_{2})$ then there exists some $k \in \mathbb{Z}_{\geq0}$ such that $s^{-(k-1)}\beta_{1}, s^{-(k-1)}\beta_{2} \in \Delta_{s^{-1}}$ which implies $d(\beta_{1}) = d(\beta_{2})$. For (ii), observe that $\text{row}(\eta_{1}) = \text{row}(\eta_{2})$ for all $\eta_{1}, \eta_{2} \in \Delta_{s^{-1}}\cap\Delta^{R}$. The result now follows by using a similar argument to (i). \\
For (iii), let $\alpha \in \Delta_{Z}$ and $\beta \in \ol{\Delta}_{K}^{k}\cap\Delta^{C}$ such that $\alpha + \beta \in \Delta$. Then $s^{-(k-1)}(\alpha + \beta) = \alpha + s^{-(k-1)}\beta$, thus $\text{row}(\alpha + \beta) = \text{row}(s^{-(k-1)}(\alpha+\beta))$ which implies $\alpha + \beta \in \Delta^{C}$. Moreover, observe that $\text{col}(\alpha+\beta) = \text{col}(\beta)$ which implies $d(\alpha+\beta) = k$ by (i). One can use a similar argument to deduce (iv). \\
Finally, as $\text{row}(\alpha) \neq \text{col}(\eta)$ or $\text{col}(\alpha) \neq \text{row}(\eta)$ for any $\alpha \in \Delta_{Z}$, $\eta \in \Delta^{O}$ it follows that $\alpha + \eta \notin \Delta$ which proves (v).
\end{proof}

\begin{lemma} \label{lemma:4.1b}
Let $1 \leq k \leq D+1$ and $\beta, \delta \in \ol{\Delta}_{K}^{k}$. If $\beta + \delta \in \Delta$ then $\beta \in \Delta^{C}$, $\delta \in \Delta^{R}$ and $\beta+\delta \in \Delta_{3}^{O}$.
\end{lemma}

\begin{proof}
If $\beta, \delta \in \ol{\Delta}_{K}^{k}$ such that $\beta+\delta \in \Delta$ then $\beta+\delta \in \ol{\Delta}_{K}^{k}$ and thus $s^{-(k-1)}\beta$, $s^{-(k-1)}\delta$, $s^{-(k-1)}\beta+s^{-(k-1)}\delta \in \Delta_{s^{-1}}$. This forces $s^{-(k-1)}\beta, s^{-(k-1)}\delta \in \Delta^{R}\cup\Delta^{C}$ and without loss of generality one can assume $s^{-(k-1)}\beta \in \Delta^{C}$ and $s^{-(k-1)}\delta \in \Delta^{R}$. The result follows.
\end{proof}

\begin{lemma} \label{lemma:4.1c}
Let $2 \leq k \leq D+1$, $\eta \in \ol{\Delta}_{K}^{k} \cap (\Delta^{C} \cup \Delta^{R})$ and $\eta_{1}, \eta_{2} \in (\ol{\Delta}_{K})_{+}$ such that $\eta = \eta_{1} + \eta_{2}$, then $\eta_{1},\eta_{2} \notin \ol{\Delta}_{K}^{k}$.
\end{lemma}

\begin{proof}
Let $\eta \in \ol{\Delta}_{K}^{k}\cap\Delta^{C}$. If $\eta=\eta_{1}+\eta_{2}$ then $\eta_{1},\eta_{2}\in\Delta^{C}\cup\Delta_{1}^{O}$ and can assume, without loss of generality, that $\eta_{1} \in \Delta^{C}$ and $\eta_{2} \in \Delta_{1}^{O}$. As $\eta_{1} \in \Delta^{C}$ implies $\text{col}(\eta) \neq \text{col}(\eta_{1})$ thus $\eta_{1} \notin \ol{\Delta}_{K}^{k}$ by Lemma \ref{lemma:4.1a} (i). Suppose $\eta_{2} \in \ol{\Delta}_{K}^{k}$, then $s^{-(k-1)}\eta = s^{-(k-1)}\eta_{1}+s^{-(k-1)}\eta_{2} \in \Delta_{s^{-1}} \cap \Delta^{C}$ and $s^{-(k-1)}\eta_{2} \in \Delta_{s^{-1}}$. As $s^{-(k-1)}\eta \in \Delta^{C}$ forces $s^{-(k-1)}\eta_{2} \in \{ \alpha_{m+p+2}, \gamma' \}$. If $s^{-(k-1)}\eta_{2} = \alpha_{m+p+2}$ then $s^{-(k-1)}\eta_{1} \in \Delta_{s}\cap\Delta^{C}$ which implies $\eta_{1} \in \Delta_{-}$, a contradiction. If $s^{-(k-1)}\eta_{2} = \gamma'$ then $\eta_{2} \in \mathcal{O}_{\gamma'} \cap (\ol{\Delta}_{K})_{+}$, a contradiction as $\mathcal{O}_{\gamma'} \cap (\ol{\Delta}_{K})_{+} \subset \Delta_{3}^{O}$. A similar argument can be used in the case $\eta \in \ol{\Delta}_{K}^{k} \cap \Delta^{R}$.
\end{proof}

\begin{lemma} \label{lemma:4.1d}
Let $2 \leq k \leq D+1$, $\eta, \eta_{1} \in \ol{\Delta}_{K}^{k}$ and $\eta_{2} \in (\ol{\Delta}_{K})_{+}$ such that $\eta = \eta_{1} + \eta_{2}$, then $\eta \in \Delta^{O}$.
\end{lemma}

\begin{proof}
Follows from Lemma \ref{lemma:4.1c}. 
\end{proof}

\begin{lemma} \label{lemma:4.1e}
There exists two distinct orbits $\mathcal{O}_{1}, \mathcal{O}_{2}$ in $(\ol{\Delta}_{K})_{+}$ with respect to the action of $\langle s \rangle$ such that if $\eta \in (\mathcal{O}_{1}\cup\mathcal{O}_{2})\cap(\ol{\Delta}_{K})_{+}$ and $\eta = \eta_{1}+\eta_{2}$ then $\eta_{1}, \eta_{2} \in \Delta_{s}\cup\Delta_{s^{-1}}$. Moreover, $\{ \eta_{1}, \eta_{2} \} \nsubseteq \Delta_{s}$ and $\{ \eta_{1}, \eta_{2} \} \nsubseteq \Delta_{s^{-1}}$.
\end{lemma}

\begin{proof}
Without loss of generality, one has $\alpha_{m} + \cdots + \alpha_{m+p+1} \in \mathcal{O}_{1}$ and $\alpha_{m+1} + \cdots + \alpha_{m+p+2} \in \mathcal{O}_{2}$. It can be easily checked that the roots $\alpha_{m} + \cdots + \alpha_{m+p+1}, \alpha_{m+1} + \cdots + \alpha_{m+p+2}$ satisfy the statements of the lemma. Moreover, if $\eta \in (\mathcal{O}_{1}\cup\mathcal{O}_{2})\cap(\ol{\Delta}_{K})_{+}$ such that $\eta = \eta_{1}+\eta_{2}$ and $\eta \notin \{ \alpha_{m} + \cdots + \alpha_{m+p+1}, \alpha_{m+1} + \cdots + \alpha_{m+p+2} \}$ then $\eta$ is necessarily the sum of two simple roots which again implies the statements of the lemma.
\end{proof}
\hfill

\subsection{} \label{sect: 4.8}
Using the partition of $(\ol{\Delta}_{K})_{+}$ given by (\ref{eq:4.1a}) the unipotent radical $N$ can be decomposed as
\begin{equation} \label{eq:4.1c}
N = N_{D+1}N_{D} \cdots N_{1}
\end{equation}
where, for each $1 \leq k \leq D+1$, the subgroup $N_{k}$ is generated by the one-parameter subgroups $X_{\alpha}$ for all $\alpha \in \ol{\Delta}^{k}_{K}$. Say that an $\alpha$-one-parameter subgroup has degree $k$ if $\alpha \in \ol{\Delta}^{k}_{K}$. Similarly, using the partition (\ref{eq:4.1b}) it follows that there exists a subset $Z'' \subseteq Z$ given by
\begin{equation} \label{eq:4.8}
Z'' = Z_{m+2} \cdots Z_{m+p+1}
\end{equation}
where, for each $m+2 \leq i \leq m+p+1$, the subgroup $Z_{i}$ is generated by the one-parameter subgroups $X_{\alpha}$ for all $\alpha \in \Delta_{Z}^{i}$. Denote by $Z_{H}$ the maximal torus of $Z$ and let $Z' = Z''Z_{H}$, then $Z'$ is just a conjugate of the big cell of $Z$ and is therefore a dense open subset. \\
Observe that $N_{s} = N_{D+1}N_{D}'$ where $N_{D}' = N_{D} \cap N_{s}$ and that this gives a decomposition of $N$ as $N = N_{s}N'_{s}$ where $N'_{s} = N \cap s^{-1}Ns$. It follows that the subgroup $NZ's^{-1}N$ can be written as
$$
NZ's^{-1}N = N_{s}N'_{s}s^{-1}NZ' = N_{s}s^{-1}NZ' = N_{s}Z's^{-1}N. 
$$
Fix some $u_{s}z's^{-1}u \in N_{s}Z's^{-1}N$. This element can be conjugated to an element of $Z's^{-1}N$, namely, by $u_{s}$ to the element $z's^{-1}v$ where $v = uu_{s} \in N$. Theorem \ref{prop:3.2a} states that there exists $n \in N$ and $n_{s} \in N_{s}$ such that 
$$
nz's^{-1}vn^{-1} = n_{s}z's^{-1}
$$
and, moreover, that such $n$ is unique. Write $n = m_{s}m$ where $m_{s} \in N_{s}$ and $m \in N'_{s}$. This gives
\begin{align*}
&n_{s}z's^{-1}n = m_{s}mz's^{-1}v, \\
&n_{s}z's^{-1}n = m_{s}z's^{-1}sz'^{-1}mz's^{-1}v, \\
&(sz'^{-1}n_{s}z's^{-1})n = (sz'^{-1}m_{s}z's^{-1})(sz'^{-1}mz's^{-1}v).
\end{align*}
From this we deduce that $n_{s} = m_{s}$ and 
\begin{equation} \label{eq:4.1d}
n = sz'^{-1}mz's^{-1}v.
\end{equation}
Using decomposition (\ref{eq:4.1c}) the equation (\ref{eq:4.1d}) can be written 
\begin{equation} \label{eq:4.1e}
n_{D+1} \cdots n_{1} = sz'^{-1}n''_{D}n_{D-1} \cdots n_{1}z's^{-1}v_{D+1} \cdots v_{1}
\end{equation} \label{eq:4.1f}
where $n_{k}, v_{k} \in N_{k}$, for each $1 \leq k \leq D+1$, and $n_{D}'' \in N_{D} \cap s^{-1}Ns$. The element $n$, in particular the element $n_{s}$, can now be found inductively. From equation (\ref{eq:4.1e}) it immediately follows that $n_{1} = v_{1}$ where $v_{1}$ is assumed to be known and 
\begin{equation} \label{eq:4.1g}
n_{k} = (sz'^{-1}n_{k-1}n_{k-2} \cdots n_{1}z's^{-1}v_{D+1} \cdots v_{k})_{k}
\end{equation}
by which we mean the $k$-th component of this expression with respect to the decomposition (\ref{eq:4.1c}). \\

\subsection{}
Define $\Sigma'_{s} = N_{s}Z's^{-1}$ and consider the morphisms 
$$
N_{s}Z's^{-1}N \overset{\nu}{\longrightarrow} N \times \Sigma'_{s} \overset{\pi}{\longrightarrow} \Sigma'_{s}
$$
where $\nu$ is the conjugation isomorphism established in Theorem \ref{prop:3.2a} and $\pi$ is projection onto the second factor. The pullback of these morphisms are the unique algebra morphisms 
$$
\mathbb{C}[\Sigma'_{s}] \overset{\pi^{*}}{\longrightarrow} \mathbb{C}[N \times \Sigma'_{s}] \overset{\nu^{*}}{\longrightarrow} \mathbb{C}[N_{s}Z's^{-1}N]
$$
defined by $\pi^{*}(f) = f \circ \pi$ and $\nu^{*}(g) = g \circ \nu$ where composition induces an isomorphism $\nu^{*}\circ\pi^{*} : \mathbb{C}[\Sigma'_{s}] \rightarrow \mathbb{C}[N_{s}Z's^{-1}N]^{N}$ of algebras. More precisely, if $\nu(u_{s}z's^{-1}u) = (x, n_{s}z's^{-1})$ then $(\nu^{*}\circ\pi^{*})(f)(u_{s}z's^{-1}u) = f(n_{s}z's^{-1})$. Next, consider the diagram
$$
\begin{tikzcd}
N_{s}Z's^{-1}N \arrow[rd] \arrow[r] & Z's^{-1}N \arrow[d] \\
                           & N \times \Sigma'_{s}.
\end{tikzcd}
$$
The map $N_{s}Z's^{-1}N \rightarrow Z's^{-1}N$ is conjugation defined by mapping a general element $u_{s}z's^{-1}u$ to $u_{s}^{-1}u_{s}z's^{-1}uu_{s}$ and the map $Z's^{-1}N \rightarrow N \times \Sigma'_{s}$ is the conjugation isomorphism $\nu$ restricted to $Z's^{-1}N$. The diagram is clearly not commutative but the composition $N_{s}Z's^{-1}N \rightarrow Z's^{-1}N \rightarrow N \times \Sigma'_{s}$ and the isomorphism $N_{s}Z's^{-1}N \rightarrow N \times \Sigma'_{s}$ do agree on the second factor. Indeed, if $u_{s}z's^{-1}u \mapsto z's^{-1}uu_{s} \mapsto (n^{-1}, n_{s}z's^{-1})$ for some unique $(n^{-1}, n_{s}z's^{-1}) \in N \times \Sigma'_{s}$ then $u_{s}z's^{-1}u \mapsto (u_{s}n^{-1}, n_{s}z's^{-1})$ under the isomorphism $\nu$. It follows that we may use the inductive method outlined in \ref{sect: 4.8} to find the element $n_{s}$ which will in turn allow us to explicitly describe the isomorphism $\nu^{*} \circ \pi^{*}$. \\

\subsection{} \label{sect: 4.10}
Define a binary relation $\preceq$ on the set $(\ol{\Delta}_{K})_{+}$ by $\alpha \preceq \alpha'$ if and only if 
\begin{enumerate}
\item[(i)] for $d(\alpha) \neq d(\alpha')$ one has $d(\alpha) > d(\alpha')$,
\item[(ii)] for $d(\alpha) = d(\alpha')$ and $\text{row}(\alpha) \neq \text{row}(\alpha')$ one has $\text{row}(\alpha) < \text{row}(\alpha')$,
\item[(iii)] for $d(\alpha) = d(\alpha')$ and $\text{row}(\alpha) = \text{row}(\alpha')$ one has $\text{col}(\alpha) \geq \text{col}(\alpha')$.
\end{enumerate} 
It is easily checked that $\preceq$ is a total order on $(\ol{\Delta}_{K})_{+}$. Elements of $N$ will be ordered according to $\preceq$, by which we mean, for $n = \prod_{\alpha \in (\ol{\Delta}_{K})_{+}} X_{\alpha}(c_{\alpha}) \in N$ the one-parameter subgroup $X_{\alpha}(c_{\alpha})$ will be written before $X_{\alpha'}(c_{\alpha'})$ in the product $n$ if and only if $\alpha \preceq \alpha'$. Observe that this respects the decomposition given by (\ref{eq:4.1c}). \\
Define a binary relation $\preceq_{Z}$ on the set $\Delta_{Z}$ by $\alpha \preceq_{Z} \alpha'$ if and only if
\begin{enumerate}
\item[(i)] for $\text{col}(\alpha) \neq \text{col}(\alpha')$ one has $\text{col}(\alpha) < \text{col}(\alpha')$,
\item[(ii)] for $\text{col}(\alpha) = \text{col}(\alpha')$ one has $\text{row}(\alpha) \leq \text{row}(\alpha')$.
\end{enumerate} 
It is easily checked that $\preceq_{Z}$ is a total order on $\Delta_{Z}$. Elements of $Z''$ will be ordered according to $\preceq_{Z}$, by which we mean, for $z'' = \prod_{\alpha \in \Delta_{Z}} X_{\alpha}(z_{\alpha}) \in Z''$ the one-parameter subgroup $X_{\alpha}(z_{\alpha})$ will be written before $X_{\alpha'}(z_{\alpha'})$ in the product $z''$ if and only if $\alpha \preceq_{Z} \alpha'$. Observe that this respects the decomposition given by (\ref{eq:4.8}). Finally, for any $z' \in Z'$, we write $z'=z''z_{H}$ where $z'' \in Z''$ and $z_{H} \in Z_{H}$. \\

\subsection{}
Let $\mathscr{X} = \{ X_{\alpha}(c_{\alpha}) : \alpha \in (\ol{\Delta}_{K})_{+}, \ c_{\alpha} \in \mathbb{C} \}$ and $\mathscr{N} \subseteq \mathscr{X}$ be any subset. Define a binary relation $\sim$ on $\mathscr{N}$ by $X_{\alpha}(c_{\alpha}) \sim X_{\alpha'}(c_{\alpha'})$ if and only if $\alpha = \alpha'$, for all $X_{\alpha}(c_{\alpha}), X_{\alpha'}(c_{\alpha'}) \in \mathscr{N}$. This defines an equivalence relation on $\mathscr{N}$ and we denote the set of equivalence classes by $\ol{\mathscr{N}}$. The binary relation $\preceq$ defined in \ref{sect: 4.10} induces a total preorder on $\mathscr{N}$ and a total order on $\ol{\mathscr{N}}$. We will abuse notation and denote both the induced preorder and induced order again by $\preceq$, that is, $X_{\alpha}(c_{\alpha}) \preceq X_{\alpha'}(c_{\alpha'})$ if and only if $\alpha \preceq \alpha'$, for all $X_{\alpha}(c_{\alpha}), X_{\alpha'}(c_{\alpha'}) \in \mathscr{N}$ and $[X_{\alpha}(c_{\alpha})]_{\sim} \preceq [X_{\alpha'}(c_{\alpha'})]_{\sim}$ if and only if $\alpha \preceq \alpha'$, for all $[X_{\alpha}(c_{\alpha})]_{\sim}, [X_{\alpha'}(c_{\alpha'})]_{\sim} \in \ol{\mathscr{N}}$. From this point onward we will write an equivalence class $[X_{\alpha}(c_{\alpha})]_{\sim} \in \ol{\mathscr{N}}$ as $[X_{\alpha}(c_{\alpha})]_{\sim} = X_{\alpha}(\sum c_{\alpha})$ where this sum is over the set $\{ c_{\alpha} : X_{\alpha}(c_{\alpha}) \in [X_{\alpha}(c_{\alpha})]_{\sim} \}$. \\
Let $u = X_{\eta_{1}}(c_{\eta_{1}}) \cdots X_{\eta_{t}}(c_{\eta_{t}})$, such that $\eta_{1}, \ldots, \eta_{t} \in (\ol{\Delta}_{K})_{+}$, be any product of one-parameter subgroups and define $\mathscr{X}(u)$ to be the set of one-parameter subgroups that appear in $u$ in the given presentation, that is, $\mathscr{X}(u) = \{ X_{\eta_{1}}(c_{\eta_{1}}), \ldots, X_{\eta_{t}}(c_{\eta_{t}}) \} \subseteq \mathscr{X}$. For $1 \leq k \leq D+1$ and $\alpha \in (\ol{\Delta}_{K})_{+}$ define the sets
\begin{align*}
\mathscr{X}_{k}^{C}(u) &= \{X_{\eta}(c_{\eta}) \in \mathscr{X}(u): \eta \in \ol{\Delta}_{K}^{k} \cap \Delta^{C} \}, \\
\mathscr{X}_{k}^{R}(u) &= \{X_{\eta}(c_{\eta}) \in \mathscr{X}(u): \eta \in \ol{\Delta}_{K}^{k} \cap \Delta^{R} \}, \\
\mathscr{X}_{k}^{O}(u) &= \{X_{\eta}(c_{\eta}) \in \mathscr{X}(u): \eta \in \ol{\Delta}_{K}^{k} \cap \Delta^{O} \}, \\
\mathscr{X}_{k}(u) &= \mathscr{X}_{k}^{C}(u)\cup\mathscr{X}_{k}^{R}(u)\cup\mathscr{X}_{k}^{O}(u), \\
\mathscr{X}_{\alpha}(u) &= \{X_{\eta}(c_{\eta}) \in \mathscr{X}(u): \eta = \alpha \}. 
\end{align*}
Observe that, for each $1 \leq k \leq D+1$, the term $n_{k} = \prod_{\alpha \in \ol{\Delta}_{K}^{k}} X_{\alpha}(c_{\alpha}) \in N_{k}$ ordered according to \ref{sect: 4.10} contains the expressions $\prod_{\alpha \in \ol{\Delta}_{K}^{k}\cap\Delta^{C}} X_{\alpha}(c_{\alpha})$ and $\prod_{\alpha \in \ol{\Delta}_{K}^{k}\cap\Delta^{R}} X_{\alpha}(c_{\alpha})$. Denote by $N_{k}^{C}$ and $N_{k}^{R}$ the subgroups in $N$ generated by the one-parameter subgroups $X_{\beta}$ and $X_{\delta}$, respectively, where $\beta \in \ol{\Delta}_{K}^{k}\cap\Delta^{C}$ and $\delta \in \ol{\Delta}_{K}^{k}\cap\Delta^{R}$. \\

\subsection{}
With presentations of elements in $Z'$ and $N$ established in \ref{sect: 4.10} we now consider the behaviour of the conjugation map $\tau : Z' \times N \rightarrow N$. By Lemma \ref{lemma:4.1a} (v) we need only consider this map when restricted to the subsets $Z' \times N_{k}^{C}$ and $Z' \times N_{k}^{R}$, for each $1 \leq k \leq D+1$. To do this we use the Baker-Campbell-Hausdorff theorem in type A which states that, for any $\alpha, \alpha' \in \Delta$,
$$
X_{\alpha}(c_{\alpha})X_{\alpha'}(z_{\alpha'}) = \left\{\begin{array}{ll}
X_{\alpha'}(z_{\alpha'})X_{\alpha}(c_{\alpha})X_{\alpha + \alpha'}(c_{\alpha}z_{\alpha'}), & \text{if $\alpha + \alpha' \in \Delta$} \\
X_{\alpha'}(z_{\alpha'})X_{\alpha}(c_{\alpha}), & \text{if $\alpha + \alpha' \notin \Delta$},
\end{array}\right.
$$
for any $c_{\alpha}, z_{\alpha'} \in \mathbb{C}$. It follows from this and Lemma \ref{lemma:4.1a} (iii), (iv) that $Z'$ normalizes the subgroups $N_{k}^{C}$ and $N_{k}^{R}$. Moreover, as $\beta + \beta', \delta + \delta' \notin \Delta$, for any $\beta,\beta' \in \ol{\Delta}_{K}^{k}\cap\Delta^{C}$ and any $\delta, \delta' \in \ol{\Delta}_{K}^{k}\cap\Delta^{R}$ one has, for any $n_{k}^{C} = \prod_{\beta \in \ol{\Delta}_{K}^{k}\cap\Delta^{C}} X_{\beta}(c_{\beta}) \in N_{k}^{C}$ and any $n_{k}^{R} = \prod_{\delta \in \ol{\Delta}_{K}^{k}\cap\Delta^{R}} X_{\delta}(c_{\delta}) \in N_{k}^{R}$, 
\begin{align*}
\tau(z'^{-1},n_{k}^{C}) &= \prod_{\beta' \in \ol{\Delta}_{K}^{k}\cap\Delta^{C}} X_{\beta'}(\sum_{\beta \in \ol{\Delta}_{K}^{k}\cap\Delta^{C}}c_{\beta}Z_{\beta'}(\beta)), \\
\tau(z'^{-1},n_{k}^{R}) &= \prod_{\delta' \in \ol{\Delta}_{K}^{k}\cap\Delta^{R}} X_{\delta'}(\sum_{\delta \in \ol{\Delta}_{K}^{k}\cap\Delta^{R}}c_{\delta}Z_{\delta'}(\delta)), \\
\end{align*}
for some $Z_{\beta'}(\beta), Z_{\delta'}(\delta) \in \mathbb{C}[z_{i,j} : m+2 \leq i,j \leq m+p+1]/(z_{m+2,m+2} \cdots z_{m+p+1,m+p+1}-1)$. \\
For $m+2 \leq i, j \leq m+p+1$ and $r \in \mathbb{Z}_{\geq 0}$ define expressions $T_{i,j}^{(r)}, S_{i,j}^{(r)}$ in the algebra $\mathbb{C}[z_{i,j} : m+2 \leq i,j \leq m+p+1]/(z_{m+2,m+2} \cdots z_{m+p+1,m+p+1}-1)$ by
\begin{equation*}
T_{i,j}^{(r)} = \left\{\begin{array}{ll} 
z_{i,i}^{-1}\sum_{i \neq j_{1}>\cdots>j_{r}>j} z_{i,j_{1}}z_{j_{1},j_{2}} \cdots z_{j_{r-1},j_{r}}z_{j_{r},j}, & \text{if $r > 0$} \\
z_{i,i}^{-1}z_{i,j}, & \text{if $r = 0$, $i \neq j$} \\
z_{i,i}^{-1}, & \text{if $r = 0$, $i = j$}
\end{array}\right.
\end{equation*}
and
\begin{equation*}
S_{i,j}^{(r)} = \left\{\begin{array}{ll} 
z_{i,i}\sum_{j \neq j_{1}<\cdots<j_{r}<i} z_{j,j_{1}}z_{j_{1},j_{2}} \cdots z_{j_{r-1},j_{r}}z_{j_{r},i}, & \text{if $r > 0$} \\
z_{j,i}z_{i,i}, & \text{if $r = 0$, $i \neq j$} \\
z_{i,i}, & \text{if $r=0$, $i = j$}.
\end{array}\right.
\end{equation*}
For notational convenience denote the elements of $\ol{\Delta}_{K}^{k}\cap\Delta^{C}$ by $\beta_{i}^{N_{k}}$ where $m+2 \leq i \leq m+p+1$ and $\text{row}(\beta_{i}^{N_{k}}) = i$. Similarly, denote the elements of $\ol{\Delta}_{K}^{k}\cap\Delta^{R}$ by $\delta_{i}^{N_{k}}$ where $m+2 \leq i \leq m+p+1$ and $\text{col}(\beta_{i}^{N_{k}}) = i$. The following lemma gives explicit descriptions of the terms $Z_{\beta'}(\beta)$, $Z_{\delta'}(\delta)$.

\begin{lemma} \label{lemma:4.2b}
Let $1 \leq k \leq D$, $z' \in Z'$, $n_{k}^{C} \in N_{k}^{C}$ and $n_{k}^{R} \in N_{k}^{R}$. Then for all $m+2 \leq i,j \leq m+p+1$ one has
$$
Z_{\beta_{i}^{N_{k}}}(\beta_{j}^{N_{k}}) = \sum_{r \geq 0} T_{i,j}^{(r)}, \quad Z_{\delta_{i}^{N_{k}}}(\delta_{j}^{N_{k}}) = \sum_{r \geq 0} S_{i,j}^{(r)}.
$$ 
\end{lemma}

\begin{proof}
The proof is by induction and is straightforward. To prove the first equality of the lemma one starts by considering the term $X_{\beta_{j}^{N_{k}}}(c_{\beta_{j}^{N_{k}}})z'$, for some $m+2 \leq j \leq m+p+1$, in $N_{k}^{C}Z'$ and the effect of successively moving the one-parameter subgroups in the product $z'$ to the left past $X_{\beta_{j}^{N_{k}}}(c_{\beta_{j}^{N_{k}}})$. One obtains
$$
X_{\beta_{j}^{N_{k}}}(c_{\beta_{j}^{N_{k}}})z' = z'\Big(\prod_{i=m+2}^{m+p+1}X_{\beta_{i}^{N_{k}}}(c_{\beta_{j}^{N_{k}}}\sum_{r \geq 0} T_{i,j}^{(r)}))\Big).
$$
Recall that $Z'$ normalizes $N_{k}^{C}$, thus the term on the right hand side of the above equality is an element of $Z'N_{k}^{C}$.  It follows that successively moving the one-parameter subgroups in $z'$ in the product $n_{k}^{C}z'$ to the left of the term $n_{k}^{C}$ gives
\begin{align*}
\Big(\prod_{j=m+2}^{m+p+1}X_{\beta_{j}^{N_{k}}}(c_{\beta_{j}^{N_{k}}})\Big)z' =& z'\Big(\prod_{j=m+2}^{m+p+1}\Big(\prod_{i=m+2}^{m+p+1}X_{\beta_{i}^{N_{k}}}(c_{\beta_{j}^{N_{k}}}\sum_{r \geq 0} T_{i,j}^{(r)})\Big)\Big) \\
=& z'\Big(\prod_{i=m+2}^{m+p+1}X_{\beta_{i}^{N_{k}}}(\sum_{j=m+2}^{m+p+1}(c_{\beta_{j}^{N_{k}}}\sum_{r \geq 0} T_{i,j}^{(r)}))\Big)
\end{align*}
and proves the first claim of the lemma. The second equality is proved in a similar way.
\end{proof}
\hfill

\subsection{} \label{sect: 4.13}
It is well known that
$$
s_{\alpha}X_{\eta}(c_{\eta})s_{\alpha}^{-1} = X_{s_{\alpha}\eta}(\theta_{\alpha, \eta}c_{\eta}),
$$
for some $c_{\eta} \in \CC$ and any $\alpha, \eta \in \Delta$, where $\theta_{\alpha, \eta} = \pm 1$ and is defined by $s_{\alpha}E_{\eta} = \theta_{\alpha, \eta}E_{s_{\alpha}\eta}$; see for example \cite[Lemma 7.2.1]{Car2}. It follows that for $s \in W$ and $\eta \in (\ol{\Delta}_{K})_{+}$ one has
$$
sX_{\eta}(c_{\eta})s^{-1} = X_{s\eta}(\prod_{a=1}^{l'} \theta_{\gamma_{a}, s_{\gamma_{a+1}} \cdots s_{\gamma_{l'}}(\eta)}c_{\eta}).
$$
Define $t(\eta) = \prod_{a=1}^{l'} \theta_{\gamma_{a}, s_{\gamma_{a+1}} \cdots s_{\gamma_{l'}}(\eta)}$ and note that $t(\eta) = \pm 1$ for all $\eta \in (\ol{\Delta}_{K})_{+}$. The term $n'_{k}$ can now be written
$$
n'_{k} = \prod_{\alpha \in \ol{\Delta}_{K}^{k}} X_{\alpha}(c'_{s^{-1}\alpha})
$$
where
$$
c'_{s^{-1}\alpha} = \left\{\begin{array}{ll}
t(\beta_{i}^{N_{k-1}})\sum_{j=m+2}^{m+p+1}c_{\beta_{j}^{N_{k-1}}}Z_{\beta_{i}^{N_{k-1}}}(\beta_{j}^{N_{k-1}}), & \text{if $s^{-1}\alpha = \beta_{i}^{N_{k-1}} \in \ol{\Delta}_{K}^{k-1} \cap \Delta^{C}$} \\
t(\delta_{i}^{N_{k-1}})\sum_{j=m+2}^{m+p+1}c_{\delta_{j}^{N_{k-1}}}Z_{\delta_{i}^{N_{k-1}}}(\delta_{j}^{N_{k-1}}), & \text{if $s^{-1}\alpha = \delta_{i}^{N_{k-1}} \in \ol{\Delta}_{K}^{k-1} \cap \Delta^{R}$} \\
t(s^{-1}\eta)c_{s^{-1}\eta}, & \text{if $s^{-1}\alpha = \eta \in \ol{\Delta}_{K}^{k-1} \cap \Delta^{O}$}.
\end{array}\right.
$$
It can be easily verified by the reader that for $\alpha \in \Lambda$ and $\beta \in (\ol{\Delta}_{K})_{+}$ such that $\alpha \neq \beta$ one has
$$
\theta_{\alpha, \beta} = \left\{\begin{array}{ll}
-1, & \text{if $\text{row}(\alpha) = \text{col}(\beta)$} \\
1, & \text{if $\text{col}(\alpha) = \text{row}(\beta)$} \\
-1, & \text{if $\text{row}(\alpha) = \text{row}(\beta)$} \\
1, & \text{if $\text{col}(\alpha) = \text{col}(\beta)$}
\end{array}\right.
$$
and $\theta_{\alpha, \alpha} = -1$. \\

\subsection{}
Recall the equation 
$$
n_{k} = (sz'^{-1}n_{k-1} \cdots n_{1}z's^{-1}v_{D+1} \cdots v_{k})_{k}
$$
established in (\ref{eq:4.1g}). This can now be written
\begin{equation} \label{eq:4.3a}
n_{k} = (n'_{k} \cdots n'_{2}v_{D+1} \cdots v_{k})_{k}.
\end{equation}
In the next section we rearrange the expression $n'_{k} \cdots n'_{2}v_{D+1} \cdots v_{k}$ to a product respecting the ordering of one-parameter subgroups established in \ref{sect: 4.10}. Using the inductive hypothesis that all terms on the right-hand side of (\ref{eq:4.3a}) are known one can then deduce the term $n_{k}$. In order to reach this end we must first prove some technical lemmas. \\

\subsection{}
For any $2 \leq k \leq D+1$ and any $\alpha \in \ol{\Delta}_{K}^{k}$ define the sets
\begin{IEEEeqnarray*}{rCl}
\mathcal{P}_{\alpha} &=& \{ (\eta_{1},\eta_{2}) \in (\ol{\Delta}_{K})_{+}\times(\ol{\Delta}_{K})_{+} : \eta_{1} \in \bigcup_{i=2}^{d(\alpha)-1} \ol{\Delta}_{K}^{i}, \eta_{2} \in \ol{\Delta}_{K}^{d(\alpha)} \ \text{and} \ \alpha = \eta_{1} + \eta_{2} \}, \\
\mathcal{P}'_{\alpha} &=& \{ (\eta_{1},\eta_{2}) \in (\ol{\Delta}_{K})_{+}\times(\ol{\Delta}_{K})_{+} : \eta_{1}, \eta_{2} \in \ol{\Delta}_{K}^{d(\alpha)} \ \text{and} \ \alpha = \eta_{1} + \eta_{2} \}.
\end{IEEEeqnarray*}

\begin{lemma} \label{lemma:4.3a}
Let $2 \leq q \leq f \leq k$, $\kappa \in \ol{\Delta}_{K}^{k}$, $\eta, \eta_{1} \in \ol{\Delta}_{K}^{f}$ and $\eta_{2} \in \bigcup_{i=f}^{D+1} \ol{\Delta}_{K}^{i}$ such that $\eta = \eta_{1}+\eta_{2}$.
\begin{itemize}
\item[(i)] If there exists $\eta_{3} \in (\ol{\Delta}_{K})_{+}$ such that $\eta+\eta_{3} \in \Delta$ then either $\eta+\eta_{3} \notin \ol{\Delta}_{K}^{f}$ or $d(\eta_{3}) < f$.
\item[(ii)] If $\kappa \in \Delta^{C} \cup \Delta^{R}$ then $\mathcal{P}_{\kappa} = \emptyset$.
\item[(iii)] If $\kappa \in (\ol{\Delta}_{K})_{+}\backslash\Delta_{3}^{O}$ then $\mathcal{P}'_{\kappa} = \emptyset$.
\end{itemize}
\end{lemma}

\begin{proof}
For (i) assume $\eta+\eta_{3} \in \ol{\Delta}_{K}^{f}$ and suppose $f \leq d(\eta_{3}) \leq D+1$. One has $s^{-(f-1)}\eta_{1} \in \Delta_{s^{-1}}$, $s^{-(f-1)}\eta_{2}, s^{-(f-1)}\eta_{3} \in (\ol{\Delta}_{K})_{+}$ and by Lemma \ref{lemma:4.1d} it follows that $\eta \in \Delta^{O}$. If $s^{-(f-1)}\eta_{2}, s^{-(f-1)}\eta_{3} \in \Delta_{s^{-1}}$ then $s^{-(f-1)}(\eta+\eta_{3})$ is a sum of three roots in $\Delta_{s^{-1}}$, a contradiction. If $s^{-(f-1)}\eta_{2} \in \Delta_{s^{-1}}$ and $s^{-(f-1)}\eta_{3} \in \Delta_{+}\backslash\Delta_{s^{-1}}$ then $s^{-(f-1)}\eta = \gamma'$ and thus $s^{-(f-1)}\eta_{3} \in \Delta_{-}$, a contradiction as $d(\eta_{3}) \geq f$. If $s^{-(f-1)}\eta_{2} \in \Delta_{+}\backslash\Delta_{s^{-1}}$ and $s^{-(f-1)}\eta_{3} \in \Delta_{s^{-1}}$ then $s^{-(f-1)}(\eta+\eta_{3})=\gamma'$ and thus $s^{-(f-1)}\eta, s^{-(f-1)}\eta_{3} \in \Delta^{C}\cup\Delta^{R}$, a contradiction as $s^{-(f-1)}\eta \in \Delta^{O}$. Finally, if $s^{-(f-1)}\eta_{2} \in \Delta_{+}\backslash\Delta_{s^{-1}}$ and $s^{-(f-1)}\eta_{3} \in \Delta_{+}\backslash\Delta_{s^{-1}}$ then $s^{-(f-1)}\eta_{1} \in \{ \alpha_{m}, \alpha_{m+p+2} \}$, $s^{-(f-1)}\eta \in \Delta^{C}\cup\Delta^{R}$ and $s^{-(f-1)}(\eta+\eta_{3}) = \gamma'$, contradiction as $s^{-(f-1)}\eta \in \Delta^{O}$.  \\
The statement (ii) follows from Lemma \ref{lemma:4.1d}. \\
For (iii), if $\kappa \in \Delta^{C} \cup \Delta^{R}$ then it follows from Lemma \ref{lemma:4.1d} that $\mathcal{P}'_{\kappa} = \emptyset$. If $\kappa \in \Delta_{1}^{O}\cup\Delta_{2}^{O}$ and we suppose there exists $(\eta_{1},\eta_{2}) \in \mathcal{P}'_{\kappa}$ then $s^{-(k-1)}\kappa = s^{-(k-1)}\eta_{1}+s^{-(k-1)}\eta_{2} \in \Delta_{s^{-1}}$ which implies $s^{-(k-1)}\kappa = \gamma'$ and thus $s^{-(k-1)}\eta_{1}, s^{-(k-1)}\eta_{2} \in \Delta^{C}\cup\Delta^{R}$, a contradiction.
\end{proof}

\begin{lemma} \label{lemma:4.3b}
Let $\eta, \eta_{1} \in \ol{\Delta}_{K}^{k}$ and $\eta_{2} \in (\bigcup_{i=2}^{k-1} \ol{\Delta}_{K}^{i}) \cup (\ol{\Delta}_{K}^{k} \cap \Delta^{O})$ such that $\eta = \eta_{1}+\eta_{2}$. If there exists $\eta_{3} \in (\bigcup_{i=2}^{k-1} \ol{\Delta}_{K}^{i}) \cup (\ol{\Delta}_{K}^{k} \cap \Delta^{O})$ such that $\eta + \eta_{3} \in \Delta$ then $\eta+\eta_{3} \notin \ol{\Delta}_{K}^{k}$.
\end{lemma}

\begin{proof}
Suppose $\eta + \eta_{3} \in \ol{\Delta}_{K}^{k}$, then by Lemma \ref{lemma:4.1d} $\eta + \eta_{3} \in \Delta^{O}$. One has $s^{-(k-1)}(\eta+\eta_{3}) = s^{-(k-1)}\eta_{1}+s^{-(k-1)}\eta_{2}+s^{-(k-1)}\eta_{3} \in \Delta_{s^{-1}}$ where $s^{-(k-1)}\eta_{2} \in \Delta_{s^{-1}}\cup\Delta_{-}$ and $s^{-(k-1)}\eta_{3} \in \Delta_{s^{-1}}\cup\Delta_{-}$. If $s^{-(k-1)}\eta_{2} \in \Delta_{s^{-1}}$ and $s^{-(k-1)}\eta_{3} \in \Delta_{s^{-1}}$ then $s^{-(k-1)}(\eta+\eta_{3})$ is a sum of three roots in $\Delta_{s^{-1}}$, a contradiction. If $s^{-(k-1)}\eta_{2} \in \Delta_{s^{-1}}$ and $s^{-(k-1)}\eta_{3} \in \Delta_{-}$ then $s^{-(k-1)}\eta = \gamma'$ and $s^{-(k-1)}\eta_{1}, s^{-(k-1)}\eta_{2} \in \Delta^{C}\cup\Delta^{R}$, a contradiction as $s^{-(k-1)}\eta_{2} \in \Delta^{O}$. If $s^{-(k-1)}\eta_{2} \in \Delta_{-}$ and $s^{-(k-1)}\eta_{3} \in \Delta_{s^{-1}}$ then $s^{-(k-1)}\eta_{1} + s^{-(k-1)}\eta_{3} \in \Delta_{s^{-1}}$, i.e. $s^{-(k-1)}\eta_{1} + s^{-(k-1)}\eta_{3} = \gamma'$ and thus $s^{-(k-1)}\eta_{3} \in \Delta^{C} \cup \Delta^{R}$, a contradiction as $s^{-(k-1)}\eta_{3} \in \Delta^{O}$. If $s^{-(k-1)}\eta_{2} \in \Delta_{-}$ and $s^{-(k-1)}\eta_{3} \in \Delta_{-}$ then $s^{-(k-1)}\eta_{1} = \gamma'$, $s^{-(k-1)}\eta \in \Delta^{C} \cup \Delta^{R}$ and $s^{-(k-1)}(\eta+\eta_{3}) \in \{ \alpha_{m}, \alpha_{m+p+2} \}$, a contradiction as $s^{-(k-1)}\eta \in \Delta^{O}$. This proves the result.
\end{proof}

\begin{lemma} \label{lemma:4.3c}
Let $2 \leq f'<f<k$, $\kappa \in \ol{\Delta}_{K}^{k}$, $\eta_{1} \in \ol{\Delta}_{K}^{f}$ and $\eta_{2} \in \ol{\Delta}_{K}^{k}$ such that $(\eta_{1},\eta_{2}) \in \mathcal{P}_{\kappa}$. If there exists $\eta_{3} \in \ol{\Delta}_{K}^{f'}$ and $\eta_{4} \in \ol{\Delta}_{K}^{k}$ such that $\eta_{3}+\eta_{4} \in \ol{\Delta}_{K}^{f}$ then $\eta_{1} \neq \eta_{3}+\eta_{4}$.
\end{lemma}

\begin{proof}
Suppose $\eta_{1} = \eta_{3} + \eta_{4}$, then $s^{-(k-1)}\kappa = s^{-(k-1)}\eta_{3} + s^{-(k-1)}\eta_{4} + s^{-(k-1)}\eta_{2} \in \Delta_{s^{-1}}$. As $s^{-(k-1)}\eta_{3} \in \Delta_{-}$ and $s^{-(k-1)}\eta_{2},s^{-(k-1)}\eta_{4} \in \Delta_{s^{-1}}$ implies $s^{-(k-1)}\eta_{2} + s^{-(k-1)}\eta_{4} \in \Delta_{s^{-1}}$, i.e. $s^{-(k-1)}\eta_{2} + s^{-(k-1)}\eta_{4} = \gamma'$ and thus $s^{-(k-1)}\eta_{2}, s^{-(k-1)}\eta_{4} \in \Delta^{C}\cup\Delta^{R}$. Suppose $s^{-(k-1)}\eta_{2} \in \Delta^{C}$. As $s^{-(k-1)}\kappa \in \Delta^{O}$ by Lemma \ref{lemma:4.1d} and $s^{-(k-1)}\kappa = s^{-(k-1)}\eta_{1} + s^{-(k-1)}\eta_{2}$ then either $s^{-(k-1)}\eta_{1} \in \Delta^{R}$ or $s^{-(k-1)}\eta_{1} \in \Delta_{1}^{O}$. If $s^{-(k-1)}\eta_{1} \in \Delta^{R}$ then $s^{-(k-1)}\kappa \in \Delta_{3}^{O} \cap \Delta_{s^{-1}}$ which implies $s^{-(k-1)}\kappa = \gamma'$ and thus $s^{-(k-1)}\eta_{1} \in \Delta_{s^{-1}}$, a contradiction as $s^{-(k-1)}\eta_{1} \in \Delta_{-}$. If $s^{-(k-1)}\eta_{1} \in \Delta_{1}^{O}$ then $s^{-(k-1)}\kappa \in \Delta^{C}$, a contradiction as $\kappa \in \Delta^{O}$. Similarly, suppose $s^{-(k-1)}\eta_{2} \in \Delta^{R}$ then either $s^{-(k-1)}\eta_{1} \in \Delta^{C}$ or $s^{-(k-1)}\eta_{1} \in \Delta_{2}^{O}$. If $s^{-(k-1)}\eta_{1} \in \Delta^{C}$ then $s^{-(k-1)}\kappa \in \Delta_{3}^{O} \cap \Delta_{s^{-1}}$ which implies $s^{-(k-1)}\kappa = \gamma'$ and thus $s^{-(k-1)}\eta_{1} \in \Delta_{s^{-1}}$, a contradiction as $s^{-(k-1)}\eta_{1} \in \Delta_{-}$. If $s^{-(k-1)}\eta_{1} \in \Delta_{2}^{O}$ then $s^{-(k-1)}\kappa \in \Delta^{R}$, a contradiction as $\kappa \in \Delta^{O}$.
\end{proof}
\hfill

\subsection{} \label{sect: 4.16}
For each $\alpha \in (\ol{\Delta}_{K})_{+}$ define the set
$$
\mathcal{P}(\alpha) = \{ (\eta, \eta') \in (\ol{\Delta}_{K})_{+} \times (\ol{\Delta}_{K})_{+} : \alpha = \eta + \eta' \}.
$$
It is important to make clear the following basic observation as we will use it repeatedly throughout the next section. Let $X_{\eta}(c_{\eta}) \in N$ and $u = X_{\eta_{1}}(c_{\eta_{1}}) \cdots X_{\eta_{t}}(c_{\eta_{t}})$, such that $\eta_{1}, \ldots, \eta_{t} \in (\ol{\Delta}_{K})_{+}$, be any product of one-parameter subgroups. Suppose
$$
X_{\eta}(c_{\eta})u = u'X_{\eta}(c_{\eta}),
$$
for some $u' \in N$, then $X_{\alpha}(b_{\alpha}) \in \mathscr{X}(u') \backslash \mathscr{X}(u)$ if and only if there exists $X_{\eta_{i}}(c_{\eta_{i}}) \in \mathscr{X}(u)$ such that $(\eta,\eta_{i}) \in \mathcal{P}(\alpha)$. Moreover, if $(\eta,\eta_{i}) \in \mathcal{P}(\alpha)$ then $b_{\alpha} = c_{\eta}c_{\eta_{i}}$. \\

\section{Invariants}
\subsection{}
Throughout this section we fix $2 \leq k \leq D+1$ and $\kappa \in \ol{\Delta}_{K}^{k}$. Write $v(1) = v_{D+1} \cdots v_{k}$ so that 
$$
n'_{k} \cdots n'_{2}v_{D+1} \cdots v_{k} = n'_{k} \cdots n'_{2}v(1).
$$
In this section we rearrange one-parameter subgroups in the above product to one respecting the order $\preceq$ described in \ref{sect: 4.10}. In \ref{sect: 5.2} and \ref{sect: 5.3} below we spell out two particular rearrangements of one-parameter subgroups. \\

\subsection{} \label{sect: 5.2}
Let $X_{\eta}(c_{\eta}) \in N$ and $u = X_{\eta_{1}}(c_{\eta_{1}}) \cdots X_{\eta_{t}}(c_{\eta_{t}})$, such that $\eta_{1}, \ldots, \eta_{t} \in (\ol{\Delta}_{K})_{+}$, be any product of one-parameter subgroups that satisfies $d(\eta_{i}) \geq d(\eta)$, for each $1 \leq i \leq t$. Consider moving the one-parameter subgroup $X_{\eta}(c_{\eta})$ to the right of $u$. This gives the equation 
$$
X_{\eta}(c_{\eta})u = u(1)X_{\eta}(c_{\eta}),
$$ 
for some $u(1) \in N$. Observe that if $X_{\alpha}(b_{\alpha}) \in \mathscr{X}(u(1))$ then $d(\alpha) \geq d(\eta)$. Let 
$$
\mathscr{X}_{d(\eta)}(u(1)) \backslash \mathscr{X}_{d(\eta)}(u) = \{ X_{\eta'_{1}}(b_{\eta'_{1}}), \ldots, X_{\eta'_{r}}(b_{\eta'_{r}}) \}
$$ 
and, without loss of generality, suppose $X_{\eta'_{1}}(b_{\eta'_{1}}) \preceq \cdots \preceq X_{\eta'_{r}}(b_{\eta'_{r}})$. It follows from Lemma \ref{lemma:4.1d} that $\eta'_{1}, \ldots, \eta'_{r} \in \Delta^{O}$. Consider the rearrangements
\begin{IEEEeqnarray*}{rCl}
u(1)X_{\eta}(c_{\eta}) &=& u(2)X_{\eta'_{r}}(b_{\eta'_{r}})X_{\eta}(c_{\eta}) \\
&=& u(3)X_{\eta'_{r-1}}(b_{\eta'_{r-1}})X_{\eta'_{r}}(b_{\eta'_{r}})X_{\eta}(c_{\eta}) \\
&\vdots& \\
&=& u(r+1)X_{\eta'_{1}}(b_{\eta'_{1}}) \cdots X_{\eta'_{r}}(b_{\eta'_{r}})X_{\eta}(c_{\eta}),
\end{IEEEeqnarray*}
for some $u(2), \ldots, u(r+1) \in N$. It follows from Lemma \ref{lemma:4.3a} (i) that $\mathscr{X}_{d(\eta)}(u(j)) \backslash \mathscr{X}_{d(\eta)}(u(j-1)) = \emptyset$, for each $2 \leq j \leq r+1$. The rearrangement $X_{\eta}(c_{\eta})u = u(r+1)X_{\eta'_{1}}(b_{\eta'_{1}}) \cdots X_{\eta'_{r}}(b_{\eta'_{r}})X_{\eta}(c_{\eta})$ just described above will simply be referred to as \ref{sect: 5.2}. \\

\subsection{} \label{sect: 5.3}
Now let $X_{\eta}(c_{\eta}) \in N$ and $u = X_{\eta_{1}}(c_{\eta_{1}}) \cdots X_{\eta_{t}}(c_{\eta_{t}})$, such that $\eta_{1}, \ldots, \eta_{t} \in (\ol{\Delta}_{K})_{+}$, be any product of one-parameter subgroups that satisfies $d(\eta_{i}) \leq d(\eta)$, for each $1 \leq i \leq t$. Consider moving the one-parameter subgroup $X_{\eta}(c_{\eta})$ to the left of $u$. This gives the equation 
$$
uX_{\eta}(c_{\eta}) = X_{\eta}(c_{\eta})u(1),
$$ 
for some $u(1) \in N$. Observe that if $X_{\alpha}(b_{\alpha}) \in \mathscr{X}(u(1))$ then $d(\alpha) \leq d(\eta)$. Let 
$$
\mathscr{X}_{d(\eta)}(u(1)) \backslash \mathscr{X}_{d(\eta)}(u) = \{ X_{\eta'_{1}}(b_{\eta'_{1}}), \ldots, X_{\eta'_{r}}(b_{\eta'_{r}}) \}
$$ 
and, without loss of generality, suppose $X_{\eta'_{1}}(b_{\eta'_{1}}) \preceq \cdots \preceq X_{\eta'_{r}}(b_{\eta'_{r}})$. It follows from Lemma \ref{lemma:4.1d} that $\eta'_{1}, \ldots, \eta'_{r} \in \Delta^{O}$. Consider the rearrangements
\begin{IEEEeqnarray*}{rCl}
X_{\eta}(c_{\eta})u(1) &=& X_{\eta}(c_{\eta})X_{\eta'_{1}}(b_{\eta'_{1}})u(2) \\
&=& X_{\eta}(c_{\eta})X_{\eta'_{1}}(b_{\eta'_{1}})X_{\eta'_{2}}(b_{\eta'_{2}})u(3) \\
&\vdots& \\
&=& X_{\eta}(c_{\eta})X_{\eta'_{1}}(b_{\eta'_{1}}) \cdots X_{\eta'_{r}}(b_{\eta'_{r}})u(r+1),
\end{IEEEeqnarray*}
for some $u(2), \ldots, u(r+1) \in N$. It follows from Lemma \ref{lemma:4.3b} that $\mathscr{X}_{d(\eta)}(u(j)) \backslash \mathscr{X}_{d(\eta)}(u(j-1)) = \emptyset$, for each $2 \leq j \leq r+1$. The rearrangement $uX_{\eta}(c_{\eta}) = X_{\eta}(c_{\eta})X_{\eta'_{1}}(b_{\eta'_{1}}) \cdots X_{\eta'_{r}}(b_{\eta'_{r}})u(r+1)$ just described above will simply be referred to as \ref{sect: 5.3}. \\

\subsection{} \label{sect: 5.4}
Define the following rule:
\begin{equation} \label{star}
\text{\parbox{.85\textwidth}{For each $2 \leq q \leq k$ and any $\eta, \eta' \in \ol{\Delta}_{K}^{q}$ say that $X_{\eta'}(c_{\eta'})$ is moved before $X_{\eta}(c_{\eta})$ if $X_{\eta}(c_{\eta}) \preceq X_{\eta'}(c_{\eta'})$. Moreover, each one-parameter subgroup $X_{\eta}(c_{\eta})$ is moved according to \ref{sect: 5.2}.} \tag{$\star$}}
\end{equation}

First consider the product $n'_{2}v(1)$ and move the one-parameter subgroups in the set $\mathscr{X}(n'_{2})$ to the right of $v(1)$ according to (\ref{star}). This gives the equation
$$
n'_{2}v(1)=v(2,2)n_{2}^{(2)},
$$
for some $v(2,2) \in N$. Next, move all one-parameter subgroups in $v(2,2)$ that appear in the set $\mathscr{X}_{3}(v(2,2)) \backslash \mathscr{X}_{3}(v(1))$ to the right of $v(2,2)$ according to (\ref{star}). This gives the equation 
$$
v(2,2)n_{2}^{(2)} = v(2,3)n_{3}^{(2)}n_{2}^{(2)},
$$
for some $v(2,3) \in N$ and $n_{3}^{(2)} \in N_{3}$. For $4 \leq f \leq k$, move all one-parameter subgroups in $v(2,f-1)$ that appear in the set $\mathscr{X}_{f}(v(2,f-1)) \backslash \mathscr{X}_{f}(v(1))$ to the right of $v(2,f-1)$ according to (\ref{star}). This gives the equation 
$$
v(2,f-1)n_{f-1}^{(2)} \cdots n_{2}^{(2)} = v(2,f)n_{f}^{(2)} \cdots n_{2}^{(2)},
$$
for some $v(2,f) \in N$. When $f=k$ write $v(2,k) = v(2)$. For $3 \leq q \leq k$, consider the product $n'_{q}v(q-1)$ and move the one-parameter subgroups in $n'_{q}$ according to (\ref{star}). This gives the equation
$$
n'_{q}v(q-1)=v(q,q)n_{q}^{(q)},
$$
for some $v(q,q) \in N$. Next, move all one-parameter subgroups in $v(q,q)$ that appear in the set $\mathscr{X}_{q+1}(v(q,q)) \backslash \mathscr{X}_{q+1}(v(q-1))$ to the right of $v(q,q)$ according to (\ref{star}). This gives the equation 
$$
v(q,q)n_{q}^{(q)} = v(q,q+1)n_{q+1}^{(q)}n_{q}^{(q)},
$$
for some $v(q,q+1) \in N$ and $n_{q+1}^{(q)} \in N_{q+1}$. For $q+1 \leq f \leq k$, move all one-parameter subgroups in $v(q,f-1)$ that appear in the set $\mathscr{X}_{f}(v(q,f-1)) \backslash \mathscr{X}_{f}(v(q-1))$ to the right of $v(q,f-1)$ according to (\ref{star}). This gives the equation 
$$
v(q,f-1)n_{f-1}^{(q)} \cdots n_{q}^{(q)} = v(q,f)n_{f}^{(q)} \cdots n_{q}^{(q)},
$$
for some $v(q,f) \in N$. When $f=k$ write $v(q) = v(q,k)$. The product $n'_{k} \cdots n'_{2}v(1)$ can now be rearranged as
\begin{IEEEeqnarray*}{rCl}
n'_{k} \cdots n'_{2}v(1) &=& n'_{k} \cdots n'_{3}v(2,2)n_{2}^{(2)} \\
&=& n'_{k} \cdots n'_{3}v(2,3)n_{3}^{(2)}n_{2}^{(2)} \\
&\vdots & \\
&=& n'_{k} \cdots n'_{3}v(2,k)n_{k}^{(2)} \cdots n_{2}^{(2)} \\
&=& n'_{k} \cdots n'_{3}v(2)n_{k}^{(2)} \cdots n_{2}^{(2)}.
\end{IEEEeqnarray*}
Repeating this rearrangement with respect to the terms $n'_{3}, \ldots, n'_{k}$ gives
\begin{equation} \label{eq:5.1a}
n'_{k} \cdots n'_{2}v(1) = v(k)(n_{k}^{(k)})(n_{k}^{(k-1)}n_{k-1}^{(k-1)})(n_{k}^{(k-2)}n_{k-1}^{(k-2)}n_{k-2}^{(k-2)}) \cdots (n_{k}^{(2)} \cdots n_{2}^{(2)}).
\end{equation}
Observe that, for each $2 \leq q \leq f \leq k$, the one-parameter subgroups appearing in $n_{f}^{(q)}$ can be rearranged according to $\preceq$ without any other one-parameter subgroups appearing. Define 
$$
n(q,f) = (n_{f}^{(q)} \cdots n_{q}^{(q)}) \cdots (n_{k}^{(2)} \cdots n_{2}^{(2)})
$$
and $n(q) = n(q,k)$ so that the equation (\ref{eq:5.1a}) can be written 
$$
n'_{k} \cdots n'_{2}v(1) = v(k)n(k).
$$

Define the following rule:
\begin{equation} \label{star2}
\text{\parbox{.85\textwidth}{For each $2 \leq q \leq k$ and any $\eta, \eta' \in \ol{\Delta}_{K}^{q}$ say that $X_{\eta}(c_{\eta})$ is moved before $X_{\eta'}(c_{\eta'})$ if $X_{\eta}(c_{\eta}) \preceq X_{\eta'}(c_{\eta'})$. Moreover, each one-parameter subgroup $X_{\eta}(c_{\eta})$ is moved according to \ref{sect: 5.3}.} \tag{$\star\star$}}
\end{equation}

Next, move the one-parameter subgroups in $n(k)$ that appear in $\mathscr{X}_{k}(n(k))$ to the left of $n(k)$ with respect to the rule (\ref{star2}). This gives the equation
$$
v(k)n(k) = v(k)n'n''
$$
for some $n' \in N_{k}$ and $n'' \in N$. The one-parameter subgroups appearing in the product $n'$ can be rearranged according to $\preceq$ without any other one-parameter subgroups appearing. Moreover, it follows from Lemma \ref{lemma:4.3b} that $\mathscr{X}_{k}(n'') \cup \cdots \cup \mathscr{X}_{D+1}(n'') = \emptyset$. Note that the term $v(k)$ can be written as a product $v(k) = v''v_{k}$ for some $v'' \in N$ such that $\mathscr{X}_{1}(v'') \cup \cdots \cup \mathscr{X}_{k}(v'') = \emptyset$. Thus $v(k)n(k) = v''v_{k}n'n''$. Finally, let $v_{k}n' = \ol{n}_{k}$ where $\ol{n}_{k}$ is the rearrangement of the one-parameter subgroups in product $v_{k}n'$ according to $\preceq$. This gives
$$
n_{k} = (n'_{k} \cdots n'_{2}v(1))_{k} = (v''\ol{n}_{k}n'')_{k} = \ol{n}_{k}
$$
where
$$
\ol{n}_{k} = \prod_{\kappa \in \ol{\Delta}_{K}^{k}} X_{\kappa}(\ol{c}_{\kappa})
$$
for some $\ol{c}_{\kappa} \in \mathbb{C}$. \\

\subsection{}
It is our aim to deduce the terms $\ol{c}_{\kappa}$ above. To do this we first consider an auxiliary rearrangement similar to the one described in \ref{sect: 5.2}. Let $u = X_{\eta_{1}}(c_{\eta_{1}}) \cdots X_{\eta_{t}}(c_{\eta_{t}})$, such that $\eta_{1}, \ldots, \eta_{t} \in (\ol{\Delta}_{K})_{+}$, be any product of one-parameter subgroups and $X_{\eta}(c_{\eta}) \in \mathscr{X}(u)$. Denote by $u''$ the product $u$ with the one-parameter subgroup $X_{\eta}(c_{\eta})$ removed. Say
$$
X_{\eta}(c_{\eta})u'' = u'(1)X_{\eta}(c_{\eta})
$$ 
for some $u'(1) \in N$. Let  
$$
\mathscr{X}_{d(\eta)}(u'(1)) \backslash \mathscr{X}_{d(\eta)}(u'') = \{ X_{\eta'_{1}}(c'_{\eta'_{1}}), \ldots, X_{\eta'_{r}}(c'_{\eta'_{r}}) \}
$$ 
and, without loss of generality, suppose $X_{\eta'_{1}}(c'_{\eta'_{1}}) \preceq \cdots \preceq X_{\eta'_{r}}(c'_{\eta'_{r}})$. Define $u''(1)$ to be the product $u'(1)$ with the one-parameter subgroup $X_{\eta'_{r}}(c'_{\eta'_{r}})$ removed. Next, consider the product 
$$
X_{\eta'_{r}}(c'_{\eta'_{r}})u''(1)X_{\eta}(c_{\eta}).
$$
Move the one-parameter subgroup $X_{\eta'_{r}}(c'_{\eta'_{r}})$ to the right of $u''(1)$ to give
$$
X_{\eta'_{r}}(c'_{\eta'_{r}})u''(1)X_{\eta}(c_{\eta}) = u'(2)X_{\eta'_{r}}(c'_{\eta'_{r}})X_{\eta}(c_{\eta})
$$
for some $u'(2) \in N$. Define $u''(2)$ to be the product $u'(2)$ with the one-parameter subgroup $X_{\eta'_{r-1}}(c'_{\eta'_{r-1}})$ removed. For $3 \leq j \leq r+1$ let $u'(j)$ be recursively defined by
\begin{IEEEeqnarray}{rCl}\label{aux}
&& X_{\eta'_{r-j+2}}(c'_{\eta'_{r-j+2}})u''(j-1)X_{\eta'_{r-j+3}}(c'_{\eta'_{r-j+3}}) \cdots X_{\eta'_{r}}(c'_{\eta'_{r}})X_{\eta}(c_{\eta}) \nonumber \\
&=& u'(j)X_{\eta'_{r-j+2}}(c'_{\eta_{r-j+2}})X_{\eta'_{r-j+3}}(c'_{\eta_{r-j+3}}) \cdots X_{\eta'_{r}}(c'_{\eta'_{r}})X_{\eta}(c_{\eta})
\end{IEEEeqnarray}
where $u''(j-1)$ is defined to be the product $u'(j-1)$ with the one-parameter subgroup $X_{\eta'_{r-j+2}}(c'_{\eta'_{r-j+2}})$ removed. Repeating the rearrangement (\ref{aux}) for $j = 3, \ldots, r+1$ one arrives at the equation 
\begin{IEEEeqnarray*}{rCl}
X_{\eta'_{1}}(c'_{\eta'_{1}})u''(r)X_{\eta'_{2}}(c'_{\eta'_{2}}) \cdots X_{\eta'_{r}}(c'_{\eta'_{r}})X_{\eta}(c_{\eta}) &=& u'(r+1)X_{\eta'_{1}}(c'_{\eta'_{1}}) \cdots X_{\eta'_{r}}(c'_{\eta'_{r}})X_{\eta}(c_{\eta}) \\
&=& u'(r+1)x_{\eta}X_{\eta}(c_{\eta})
\end{IEEEeqnarray*}
where $x_{\eta} = X_{\eta'_{1}}(c'_{\eta'_{1}}) \cdots X_{\eta'_{r}}(c'_{\eta'_{r}})$. \\

\subsection{}
Suppose $\mathscr{X}_{f}(v(q,f-1)) \backslash \mathscr{X}_{f}(v(q-1)) = \{ X_{\eta_{1}}(c_{\eta_{1}}), \ldots, X_{\eta_{t}}(c_{\eta_{t}}) \}$ such that $X_{\eta_{1}}(c_{\eta_{1}}) \preceq \cdots \preceq X_{\eta_{t}}(c_{\eta_{t}})$. Applying the above auxiliary rearrangement to $u = v(q,f-1)$ and $X_{\eta}(c_{\eta}) = X_{\eta_{t}}(c_{\eta_{t}})$ gives the product
$$
v(q,f-1)'(r_{t}+1)x_{\eta_{t}}X_{\eta_{t}}(c_{\eta_{t}}).
$$
Write $v(q,f-1)'_{\eta_{t}} = v(q,f-1)'(r_{t}+1)$. Next, applying the auxiliary rearrangement to $u=v(q,f-1)'(r_{t}+1)$ and $X_{\eta}(c_{\eta}) = X_{\eta_{t-1}}(c_{\eta_{t-1}})$ gives the product 
$$
v(q,f-1)'(r_{t}+1)'(r_{t-1}+1)x_{\eta_{t-1}}X_{\eta_{t-1}}(c_{\eta_{t-1}})x_{\eta_{t}}X_{\eta_{t}}(c_{\eta_{t}}).
$$
Write $v(q,f-1)'_{\eta_{t-1}} = v(q,f-1)'(r_{t}+1)'(r_{t-1}+1)$. Continuing these auxiliary rearrangements for $X_{\eta_{t-2}}(c_{\eta_{t-2}}), \ldots, X_{\eta_{1}}(c_{\eta_{1}})$ one arrives at the product
$$
v'(q,f)x_{\eta_{1}}X_{\eta_{1}}(c_{\eta_{1}}) \cdots x_{\eta_{t}}X_{\eta_{t}}(c_{\eta_{t}})
$$
where $v'(q,f) = v(q,f-1)'_{\eta_{1}} = v(q,f-1)'(r_{t}+1)' \cdots (r_{2}+1)'(r_{1}+1)$. We write
$$
x_{\eta_{1}}X_{\eta_{1}}(c_{\eta_{1}}) \cdots x_{\eta_{t}}X_{\eta_{t}}(c_{\eta_{t}}) = (n')_{f}^{(q)}
$$ 
where the right-hand side of this equation is the rearrangement of one-parameter subgroups on the left-hand side according to $\preceq$. Note that this can be done without any other one-parameter subgroups appearing. Finally, define $v'_{\eta_{t}}(q,f-1) = v'(q,f-1)$ and $v'_{\eta_{i}}(q,f-1) = v'(q,f-1)_{\eta_{i+1}}$, for each $1 \leq i \leq t-1$. \\

\subsection{}
For each $2 \leq q \leq f \leq k$ and $\eta \in (\ol{\Delta}_{K})_{+}$ we now explicitly describe the set $\mathscr{X}_{\eta}((n')_{f}^{(q)})$ from which we may easily deduce $\mathscr{X}_{\eta}(n_{f}^{(q)})$. For each root $\alpha \in \Delta$ define the subsets
\begin{align*}
\text{Row}(\alpha) &= \{ \eta \in \Delta : \text{row}(\eta) = \text{row}(\alpha),\ \text{col}(\eta)<\text{col}(\alpha) \}, \\
\text{Col}(\alpha) &= \{ \eta \in \Delta : \text{col}(\eta) = \text{col}(\alpha),\ \text{row}(\eta)>\text{row}(\alpha) \}
\end{align*}
of $\Delta$ and the subsets
\begin{IEEEeqnarray*}{rCl}
\mathcal{C}_{\eta}^{q,f} &=& \{ (\eta_{1}, \eta_{2}) \in \text{Col}(\eta) \times \text{Row}(\eta) : \eta_{1} \in \bigcup_{i=q}^{f} \ol{\Delta}_{K}^{i}, \eta_{2} \in \bigcup_{i=d(\eta)}^{D+1} \ol{\Delta}_{K}^{i}, \eta = \eta_{1} + \eta_{2} \}, \\
\mathcal{R}_{\eta}^{q,f} &=& \{ (\eta_{1}, \eta_{2}) \in \text{Row}(\eta) \times \text{Col}(\eta) : \eta_{1} \in \bigcup_{i=q}^{f} \ol{\Delta}_{K}^{i}, \eta_{2} \in \bigcup_{i=d(\eta)}^{D+1} \ol{\Delta}_{K}^{i}, \eta = \eta_{1} + \eta_{2} \}, \\
\mathcal{P}_{\eta}^{q,f} &=& \mathcal{C}_{\eta}^{q,f} \cup \mathcal{R}_{\eta}^{q,f}
\end{IEEEeqnarray*}
of $(\ol{\Delta}_{K})_{+} \times (\ol{\Delta}_{K})_{+}$, where we write $\mathcal{C}_{\eta}^{f}$, $\mathcal{R}_{\eta}^{f}$ and $\mathcal{P}_{\eta}^{f}$ when $q=f$. Observe that if $(\eta_{1}, \eta_{2}) \in \mathcal{P}_{\eta}^{q,f}$ then $\vert \mathcal{P}_{\eta_{1}}^{q,f} \vert, \vert \mathcal{P}_{\eta_{2}}^{q,f} \vert < \vert \mathcal{P}_{\eta}^{q,f} \vert$ and, moreover, one has $\mathcal{P}_{\eta}^{q,f} = \mathcal{P}_{\eta}^{q} \cup \mathcal{P}_{\eta}^{q+1} \cup \cdots \cup \mathcal{P}_{\eta}^{f}$. Clearly, the preceding sentence may be stated similarly for $\mathcal{C}_{\eta}^{q,f}$ and $\mathcal{R}_{\eta}^{q,f}$. \\
For all $2 \leq q \leq f \leq k$ and $\eta \in (\ol{\Delta}_{K})_{+}$ define, recursively, the term $B_{\eta}^{q,f}$ in the free algebra $\CC\langle c_{\alpha}, c'_{\alpha} : \alpha \in (\ol{\Delta}_{K})_{+} \rangle$ by 
$$ \arraycolsep=1.4pt\def\arraystretch{1.2}
B_{\eta}^{q,f} = \left\{\begin{array}{ll} 
B_{\eta}^{q,f-1}, & \text{if $\vert \mathcal{P}_{\eta}^{f} \vert = 0$} \\
B_{\eta}^{q,f-1} + \sum_{(\eta_{1}, \eta_{2}) \in \mathcal{C}_{\eta}^{f}} B_{\eta_{1}}^{q,f}B_{\eta_{2}}^{q,f-1} + \sum_{(\eta_{1}, \eta_{2}) \in \mathcal{R}_{\eta}^{f}} B_{\eta_{1}}^{q,f}B_{\eta_{2}}^{q,f}, & \text{if $\vert \mathcal{P}_{\eta}^{f} \vert > 0$}
\end{array}\right.
$$
where
$$
B_{\eta}^{2,2} = c'_{s^{-1}\eta} + \sum_{(\eta_{1}, \eta_{2}) \in \mathcal{P}_{\eta}^{2}} c'_{s^{-1}\eta_{1}}c_{\eta_{2}}
$$
and we define $B_{\eta}^{q,q-1} = B_{\eta}^{q-1,k}$. \\
Recall the orbits $\mathcal{O}_{1}, \mathcal{O}_{2} \subseteq (\ol{\Delta}_{K})_{+}$ defined in Lemma \ref{lemma:4.1e}, they will be used in the following lemma.

\begin{lemma} \label{lemma: 5.7.1}
Let $2 \leq f \leq k$, $\eta, \eta_{1} \in \ol{\Delta}_{K}^{f}$ and $\eta_{2} \in (\ol{\Delta}_{K})_{+}$ such that $\eta = \eta_{1} + \eta_{2}$ then $\eta_{2} \in \mathcal{O}_{1}\cup\mathcal{O}_{2}$.
\end{lemma}

\begin{proof}
The requirement that $\eta, \eta_{1} \in \ol{\Delta}_{K}^{f}$ implies $s^{-(f-1)}\eta, s^{-(f-1)}\eta_{1} \in \Delta_{s^{-1}}$ which forces $s^{-(f-1)}\eta \in s_{1}\Delta_{s_{2}}$. It now follows that $\eta_{2} \in \mathcal{O}_{1}\cup\mathcal{O}_{2}$. This can be checked on a case-by-case basis.
\end{proof}

A consequence of the above lemma is that, for each $\eta \in \ol{\Delta}_{K}^{2}$, one has $\ol{\mathscr{X}_{\eta}((n')_{2}^{(2)})} = \{ X_{\eta}(B_{\eta}^{2,2}) \}$. Indeed, it is clear that $\ol{\mathscr{X}_{\eta}((n')_{2}^{(2)})} = \ol{\mathscr{X}_{\eta}((n_{2}^{(2)}))}$ and moreover following the discussion in \ref{sect: 4.16} the $\eta$-one-parameter subgroups that appear as as result of the rearrangement $n'_{2}v(1) = v(2,2)n_{2}^{(2)}$ are precisely those given by $X_{\eta}(c'_{s^{-1}\eta_{1}}c_{\eta_{2}})$ for each $(\eta_{1},\eta_{2}) \in \mathcal{P}_{\eta}^{2}$ where $n'_{2} = \prod_{\eta \in \ol{\Delta}_{K}^{2}} X_{\eta}(c'_{s^{-1}\eta})$ and $c'_{s^{-1}\eta}$ is given in \ref{sect: 4.13}.

\begin{lemma} \label{lemma: 5.7.2}
For all $2 \leq q \leq f \leq k$ and $\eta \in \bigcup_{i=f}^{k} \ol{\Delta}_{K}^{i}$, if $\ol{\mathscr{X}_{\eta}(v'(q,f-1))} = \{ X_{\eta}(B_{\eta}^{q,f-1}) \}$ then $\ol{\mathscr{X}_{\eta}(v'(q,f)(n')_{f}^{(q)})} = \{ X_{\eta}(B_{\eta}^{q,f}) \}$.
\end{lemma}

\begin{proof}
This is a simple induction on $\vert \mathcal{P}_{\eta}^{f} \vert$. For all $\eta \in \bigcup_{i=f}^{k} \ol{\Delta}_{K}^{i}$ such that $\vert \mathcal{P}_{\eta}^{f} \vert = 0$ one clearly has $\ol{\mathscr{X}_{\eta}(v'(q,f)(n')_{f}^{(q)}))} = \ol{\mathscr{X}_{\eta}(v'(q,f-1))}$, that is, $B_{\eta}^{q,f} = B_{\eta}^{q,f-1}$. Consider the roots $\eta \in \bigcup_{i=f}^{k} \ol{\Delta}_{K}^{i}$ such that $\vert \mathcal{P}_{\eta}^{f} \vert > 0$ and suppose the claim is true for all $\eta' \in \bigcup_{i=f}^{k} \ol{\Delta}_{K}^{i}$ such that $\vert \mathcal{P}_{\eta'}^{f} \vert < \vert \mathcal{P}_{\eta}^{f} \vert$. Fix $\eta$ such that $\vert \mathcal{P}_{\eta}^{f} \vert > 0$ and let $(\eta_{1}, \eta_{2}) \in \mathcal{P}_{\eta}^{f}$. First, suppose $(\eta_{1}, \eta_{2}) \in \mathcal{C}_{\eta}^{f}$ then for all $(\eta_{3}, \eta_{4}) \in \mathcal{P}_{\eta_{1}}^{f}$ one has $\eta_{3} \succeq \eta_{1}$ thus $\ol{\mathscr{X}_{\eta_{1}}(v'(q,f))} \backslash \ol{\mathscr{X}_{\eta_{1}}(v'_{\eta_{1}}(q,f-1))} = \emptyset$. As $\vert \mathcal{P}_{\eta_{1}}^{f} \vert < \vert \mathcal{P}_{\eta}^{f} \vert$, it follows that $\ol{\mathscr{X}_{\eta_{1}}(v'_{\eta_{1}}(q,f-1))} = \{ X_{\eta_{1}}(B_{\eta_{1}}^{q,f}) \}$. For all $(\eta_{3}, \eta_{4}) \in \mathcal{P}_{\eta_{2}}^{f}$ one has $\eta_{3} \preceq \eta_{1}$ thus $\ol{\mathscr{X}_{\eta_{2}}(v'(q,f-1)_{\eta_{1}})} \backslash \ol{\mathscr{X}_{\eta_{2}}(v'_{\eta_{1}}(q,f-1))} = \emptyset$. It follows that $\ol{\mathscr{X}_{\eta_{2}}(v'(q,f-1)_{\eta_{1}})} = \{ X_{\eta_{2}}(B_{\eta_{2}}^{q,f-1}) \}$. \\
Suppose $(\eta_{1}, \eta_{2}) \in \mathcal{R}_{\eta}^{f}$, then for all $(\eta_{3}, \eta_{4}) \in \mathcal{P}_{\eta_{1}}^{f}$ one has $\eta_{3} \succeq \eta_{1}$ thus $\ol{\mathscr{X}_{\eta_{1}}(v'(q,f))} \backslash \ol{\mathscr{X}_{\eta_{1}}(v'_{\eta_{1}}(q,f-1))} = \emptyset$. As $\vert \mathcal{P}_{\eta_{1}}^{f} \vert < \vert \mathcal{P}_{\eta}^{f} \vert$, it follows that $\ol{\mathscr{X}_{\eta_{1}}(v'_{\eta_{1}}(q,f-1))} = \{ X_{\eta_{1}}(B_{\eta_{1}}^{q,f}) \}$. For all $(\eta_{3}, \eta_{4}) \in \mathcal{P}_{\eta_{2}}^{f}$ one has $\eta_{3} \succeq \eta_{1}$ thus $\ol{\mathscr{X}_{\eta_{2}}(v'(q,f))} \backslash \ol{\mathscr{X}_{\eta_{2}}(v'_{\eta_{1}}(q,f-1))} = \emptyset$ and therefore $\ol{\mathscr{X}(v'(q,f-1)_{\eta_{1}})} = \{ X_{\eta_{2}}(B_{\eta_{2}}^{q,f}) \}$. The result now follows.
\end{proof}

By what was said following Lemma \ref{lemma: 5.7.1} and by Lemma \ref{lemma: 5.7.2} above it follows that, for all $2 \leq q \leq f \leq k$ and all $\eta \in \ol{\Delta}_{K}^{f}$, one has
$$
\ol{\mathscr{X}_{\eta}((n')_{f}^{(q)})} = \mathscr{X}_{\eta}((n')_{f}^{(q)}) = \{ X_{\eta}(B_{\eta}^{q,f}) \}
$$
and thus
$$
\ol{\mathscr{X}_{\eta}(n'(k))} = \{ X_{\eta}( \sum _{q = 2}^{d(\eta)} B_{\eta}^{q,d(\eta)}) \}
$$
where $n'(k) = ((n')_{k}^{(k)})((n')_{k}^{(k-1)}(n')_{k-1}^{(k-1)}) \cdots ((n')_{k}^{(2)} \cdots (n')_{2}^{(2)})$. \\

\subsection{}
Suppose $\ol{\mathscr{X}_{\eta}(n_{f}^{(q)})} = \{ X_{\eta}(C_{\eta}^{q,f}) \}$. Let $b_{\eta_{1}} \cdots b_{\eta_{t}}$ be a summand in the expression $B_{\eta}^{q,f}$ and let $\eta_{i_{1}}, \ldots, \eta_{i_{r}}$ be the roots such that $d(\eta_{i_{j}}) \geq d(\eta)$, for each $1 \leq j \leq r$. Associate to this summand the sequence $d(\eta_{i_{1}}), \ldots, d(\eta_{i_{r}})$. \\
Observe that for any $\eta \in \ol{\Delta}_{K}^{f}$ such that $\eta = \eta_{1} + \eta_{2}$, for some $\eta_{1}, \eta_{2} \in (\ol{\Delta}_{K})_{+}$, one must have, without loss of generality, $\eta_{1} \in \bigcup_{i=1}^{f}\ol{\Delta}_{K}^{i}$ and $\eta_{2} \in \bigcup_{i=f}^{D+1}\ol{\Delta}_{K}^{i}$. It follows that, for $d(\eta_{2}) \neq d(\eta)$, one has $\eta \succeq \eta_{2}$ and, for $d(\eta_{2}) = d(\eta)$, one has $\eta \preceq \eta_{2}$. Moreover, it follows that, for $d(\eta_{2}) \neq d(\eta)$, any $\eta$-one-parameter subgroup that appears in $v(q,f)$ will appear to the left of $X_{\eta}(c_{\eta}) \in \mathscr{X}(v(q,f))$ and, for $d(\eta_{2}) = d(\eta)$, any $\eta$-one-parameter subgroup that appears in $v(q,f)$ will appear to the right of $X_{\eta}(c_{\eta}) \in \mathscr{X}(v(q,f))$. Thus, the expression $C_{\eta}^{q,f}$ is obtained from $B_{\eta}^{q,f}$ by considering only the summands $b_{\eta_{1}} \cdots b_{\eta_{t}}$ in $B_{\eta}^{q,f}$ such that $d(\eta_{i_{1}}) \leq \cdots \leq d(\eta_{i_{r}})$. \\
Define, for each $\kappa \in \ol{\Delta}_{K}^{k}$, the expression
$$
C_{\kappa} = \left\{\begin{array}{ll}
\begin{array}{l} \sum_{q=2}^{k}C_{\kappa}^{q,k},\end{array} & \text{if $\kappa \in \ol{\Delta}_{K}^{k}\cap(\Delta^{C}\cup\Delta^{R})$}\\

\begin{array}{l} \sum_{q=2}^{k}C_{\kappa}^{q,k} + \sum_{(\eta_{1},\eta_{2}) \in \mathcal{P}_{\kappa}}\sum_{q=2}^{d(\eta_{1})-1}\Big(C_{\eta_{2}}^{q,k}\Big(\sum_{q'=q+1}^{d(\eta_{1})} C_{\eta_{1}}^{q',d(\eta_{1})}\Big)\Big),\end{array} & \text{if $\kappa \in \ol{\Delta}_{K}^{k}\cap(\Delta_{1}^{O}\cup\Delta_{2}^{O})$}\\

\begin{array}{l} \sum_{q=2}^{k}C_{\kappa}^{q,k} + \sum_{(\eta_{1},\eta_{2}) \in \mathcal{P}_{\kappa}}\sum_{q=2}^{d(\eta_{1})-1}\Big(C_{\eta_{2}}^{q,k}\Big(\sum_{q'=q+1}^{d(\eta_{1})} C_{\eta_{1}}^{q',d(\eta_{1})}\Big)\Big) \\
+ \sum_{(\eta_{1},\eta_{2}) \in \mathcal{P}'_{\kappa}} \Big(\sum_{q=2}^{k}C^{q,k}_{\eta_{1}}\Big( c_{\eta_{2}} + \sum_{q'=q+1}^{k} C_{\eta_{2}}^{q',k} \Big)\Big),\end{array} & \text{if $\kappa \in \ol{\Delta}_{K}^{k}\cap\Delta_{3}^{O}$}
\end{array}\right.
$$
in the algebra $\CC[c_{\alpha}, c'_{\alpha} : \alpha \in (\ol{\Delta}_{K})_{+}]$. We can now state the main result of this article.

\begin{proposition}\label{prop: 5.8}
There is an isomorphism of algebras $\CC[NZ's^{-1}N]^{N} \simeq \CC[C_{\kappa} : \kappa \in \Delta_{s}]$. 
\end{proposition}

\begin{proof}
The statement of the lemma is proven if one can show $\ol{c}_{\kappa} = C_{\kappa}$, for each $\kappa \in \ol{\Delta}_{K}^{k}$, where $\ol{c}_{\kappa}$ was defined at the end of \ref{sect: 5.4}. Let $\kappa \in \ol{\Delta}_{K}^{k}\cap(\Delta^{C}\cup\Delta^{R})$ and consider the rearrangements
\begin{equation} \label{eq:4.3c}
v(k)n(k)= v(k)n'n'' = v''\ol{n}_{k}n''
\end{equation}
outlined in \ref{sect: 5.4}. By Lemma \ref{lemma:4.3a} (ii) and (iii) one has 
$$
\ol{\mathscr{X}_{\kappa}(v(k)n'n'')} \backslash \ol{\mathscr{X}_{\kappa}(v(k)n(k))} = \ol{\mathscr{X}_{\kappa}(v''\ol{n}_{k}n'')} \backslash \ol{\mathscr{X}_{\kappa}(v(k)n'n'')} = \emptyset
$$
thus, for such $\kappa$, it follows that
$$
\ol{c}_{\kappa} = \sum_{q=2}^{k}C_{\kappa}^{q,k} = C_{\kappa}.
$$

Let $\kappa \in \ol{\Delta}_{K}^{k}\cap\Delta^{O}$ and consider the rearrangement $v(k)n(k) = v(k)n'n''$. It follows by Lemma \ref{lemma:4.3b} and  Lemma \ref{lemma:4.3c} that
$$
\ol{\mathscr{X}_{\kappa}(n')} \backslash \ol{\mathscr{X}_{\kappa}(n(k))} = \{ X_{\kappa}(\sum_{(\eta_{1},\eta_{2}) \in \mathcal{P}_{\kappa}}\sum_{q=2}^{d(\eta_{1})-1}\Big(C_{\eta_{2}}^{q,k}\Big(\sum_{q'=q+1}^{d(\eta_{1})} C_{\eta_{1}}^{q',d(\eta_{1})}\Big)\Big)) \}.
$$
Next, consider the rearrangement $v(k)n'n'' = v''\ol{n}_{k}n''$. It follows immediately from Lemma \ref{lemma:4.3a} (iii) that if $\kappa \in \Delta_{1}^{O}\cup\Delta_{2}^{O}$ then $\ol{\mathscr{X}_{\kappa}(v''\ol{n}_{k}n'')} \backslash \ol{\mathscr{X}_{\kappa}(v(k)n'n'')} = \emptyset$ and therefore
$$
\ol{c}_{\kappa} = \sum_{q=2}^{k}C_{\kappa}^{q,k} + \sum_{(\eta_{1},\eta_{2}) \in \mathcal{P}_{\kappa}}\sum_{q=2}^{d(\eta_{1})-1}\Big(C_{\eta_{2}}^{q,k}\Big(\sum_{q'=q+1}^{d(\eta_{1})} C_{\eta_{1}}^{q',d(\eta_{1})}\Big)\Big) = C_{\kappa}.
$$

Finally, if $\kappa \in \Delta_{3}^{O}$ and considering the rearrangement $v(k)n'n'' = v''\ol{n}_{k}n''$ one has
$$
\ol{\mathscr{X}_{\kappa}(v''\ol{n}_{k}n'')} \backslash \ol{\mathscr{X}_{\kappa}(v(k)n'n'')} = \{ X_{\kappa}(\sum_{(\eta_{1},\eta_{2}) \in \mathcal{P}'_{\kappa}} \Big(\sum_{q=2}^{k}C^{q,k}_{\eta_{1}}\Big( c_{\eta_{2}} + \sum_{q'=q+1}^{k} C_{\eta_{2}}^{q',k} \Big)\Big)) \}.
$$
We therefore conclude that
\begin{IEEEeqnarray*}{rCl}
\ol{c}_{\kappa} &=& \sum_{q=2}^{k}C_{\kappa}^{q,k} + \sum_{(\eta_{1},\eta_{2}) \in \mathcal{P}_{\kappa}}\sum_{q=2}^{d(\eta_{1})-1}\Big(C_{\eta_{2}}^{q,k}\Big(\sum_{q'=q+1}^{d(\eta_{1})} C_{\eta_{1}}^{q',d(\eta_{1})}\Big)\Big) \\
&&+ \sum_{(\eta_{1},\eta_{2}) \in \mathcal{P}'_{\kappa}} \Big(\sum_{q=2}^{k}C^{q,k}_{\eta_{1}}\Big( c_{\eta_{2}} + \sum_{q'=q+1}^{k} C_{\eta_{2}}^{q',k} \Big)\Big) \\ 
&=& C_{\kappa}.
\end{IEEEeqnarray*}
\end{proof}
\hfill

\section{Appendix}
Let $\alpha \in (\ol{\Delta}_{K})_{+}$ and throughout this section write $\alpha = \alpha_{i} + \cdots + \alpha_{j}$. It is clear that 
$$
\text{Row}(\alpha)\times\text{Col}(\alpha) = \{ (\alpha_{i} + \cdots + \alpha_{k}, \alpha_{k+1} + \cdots + \alpha_{j}) : i \leq k \leq j \}
$$
However, the value of $d(\alpha)$, for all $\alpha \in (\ol{\Delta}_{K})_{+}$, is needed in order to write the expressions $C^{q,f}_{\alpha}$, for each $\alpha \in (\ol{\Delta}_{K})_{+}$. It is therefore the purpose of this appendix to explicitly describe $d(\ol{\Delta}_{K})_{+}$ for each of the four cases in Proposition \ref{prop:4.1a}. \\
Define the function $\mathscr{D} : \ZZ \times \ZZ \rightarrow \ZZ$ by $\mathscr{D}(a,b) = \lfloor \frac{1}{2}(b-a) \rfloor$. \\

\subsection{} Case (i). For $i \in \{ 1,3,5, \ldots, m+1 \}$ one has
$$\arraycolsep=1.4pt\def\arraystretch{1.2}
d(\alpha) = \left\{\begin{array}{ll}
\mathscr{D}(-\frac{1}{2}i, 2m-\frac{1}{2}j+3), & \text{if $j \in \{ i, i+2, \ldots, m-1 \}$} \\
\mathscr{D}(-\frac{1}{2}i,2m-\frac{1}{2}(j-p)+3), & \text{if $j \in \{ m+p+1, m+p+3, \ldots, m+p+m+1 \}$} \\
\mathscr{D}(-\frac{1}{2}i,2m-l+\frac{1}{2}j+5), & \text{if $j \in \{ m+p+m, m+p+m-2, \ldots, m+p+2 \}$} \\
\mathscr{D}(-\frac{1}{2}i,2m-l+\frac{1}{2}(j+p)+5), & \text{if $j \in \{ m,m-2, \ldots, i+1 \}$}
\end{array}\right.
$$
and for $i \in \{ m+p+3, m+p+5, \ldots, m+p+m+1 \}$ one has
$$\arraycolsep=1.4pt\def\arraystretch{1.2}
d(\alpha) = \left\{\begin{array}{ll}
\mathscr{D}(-\frac{1}{2}i, 2m-\frac{1}{2}j+3), & \text{if $j \in \{ i, i+2, \ldots, m+p+m+1 \}$} \\
\mathscr{D}(-\frac{1}{2}i, 2m-l+\frac{1}{2}(j+5)), & \text{if $j \in \{ m+p+m, m+p+m-2, \ldots, i+1 \}$}.
\end{array}\right.
$$
For $i \in \{ 2, 4, 6, \ldots, m \}$ one has 
$$\arraycolsep=1.4pt\def\arraystretch{1.2}
d(\alpha) = \left\{\begin{array}{ll}
\mathscr{D}(\frac{1}{2}i, \frac{1}{2}j+1), & \text{if $j \in \{ i, i+2, \ldots, m \}$} \\
\mathscr{D}(\frac{1}{2}i, \frac{1}{2}(j-p)+1), & \text{if $j \in \{ m+p+2, m+p+4, \ldots, m+p+m \}$} \\
\mathscr{D}(\frac{1}{2}i, l-\frac{1}{2}j), & \text{if $j \in \{ m+p+m+1, m+p+m-1, \ldots, m+p+1 \}$} \\
\mathscr{D}(\frac{1}{2}i, l-\frac{1}{2}(j+p)), & \text{if $j \in \{ m-1,m-3, \ldots, i+1 \}$}
\end{array}\right.
$$
and for $i \in \{ m+p+2, m+p+4, \ldots, m+p+m \}$ one has
$$\arraycolsep=1.4pt\def\arraystretch{1.2}
d(\alpha) = \left\{\begin{array}{ll}
\mathscr{D}(\frac{1}{2}i, \frac{1}{2}j+1), & \text{if $j \in \{ i, i+2, \ldots, m+p+m \}$} \\
\mathscr{D}(\frac{1}{2}i, l-\frac{1}{2}(j-5)), & \text{if $j \in \{ m+p+m+1, m+p+m-1, \ldots, i+1 \}$}.
\end{array}\right.
$$
This describes the image $d(\Delta^{O})$. For $m+2 \leq i \leq m+p+1$ one has
$$\arraycolsep=1.4pt\def\arraystretch{1.2}
d(\alpha) = \left\{\begin{array}{ll}
\mathscr{D}(m+p-1,j), & \text{if $j \in \{ m+p+2, m+p+4, \ldots, m+p+m \}$} \\
\mathscr{D}(j,m+l+3), & \text{if $j \in \{ m+p+m+1, m+p+m-1, \ldots, m+p+1 \}$}.
\end{array}\right.
$$
This describes the image $d(\Delta^{C})$. For $m+1 \leq j \leq m+p$ one has
$$\arraycolsep=1.4pt\def\arraystretch{1.2}
d(\alpha) = \left\{\begin{array}{ll}
\mathscr{D}(i,m+3), & \text{if $i \in \{ m, m-2, \ldots, 2 \}$} \\
\mathscr{D}(-i,m+2), & \text{if $i \in \{ 1, 3, \ldots, m+1 \}$}.
\end{array}\right.
$$
This describes the image $d(\Delta^{R})$. \\

\subsection{} Case (ii). For $i \in \{ 1,3,5, \ldots, m+1 \}$ one has
$$\arraycolsep=1.4pt\def\arraystretch{1.2}
d(\alpha) = \left\{\begin{array}{ll}
\mathscr{D}(-\frac{1}{2}i, 2m-\frac{1}{2}j+4), & \text{if $j \in \{ i, i+2, \ldots, m-1 \}$} \\
\mathscr{D}(-\frac{1}{2}i,2m-\frac{1}{2}(j-p)+4), & \text{if $j \in \{ m+p+1, m+p+3, \ldots, m+p+m+1 \}$} \\
\mathscr{D}(-\frac{1}{2}i,2m-l+\frac{1}{2}j+6), & \text{if $j \in \{ m+p+m+2, m+p+m, \ldots, m+p+2 \}$} \\
\mathscr{D}(-\frac{1}{2}i,2m-l+\frac{1}{2}(j+p)+6), & \text{if $j \in \{ m,m-2, \ldots, i+1 \}$}
\end{array}\right.
$$
and for $i \in \{ m+p+3, m+p+5, \ldots, m+p+m+1 \}$ one has
$$\arraycolsep=1.4pt\def\arraystretch{1.2}
d(\alpha) = \left\{\begin{array}{ll}
\mathscr{D}(-\frac{1}{2}i, 2m-\frac{1}{2}j+4), & \text{if $j \in \{ i, i+2, \ldots, m+p+m+1 \}$} \\
\mathscr{D}(-\frac{1}{2}i, 2m-l+\frac{1}{2}(j+7)), & \text{if $j \in \{ m+p+m+2, m+p+m, \ldots, i+1 \}$}.
\end{array}\right.
$$
For $i \in \{ 2, 4, 6, \ldots, m \}$ one has 
$$\arraycolsep=1.4pt\def\arraystretch{1.2}
d(\alpha) = \left\{\begin{array}{ll}
\mathscr{D}(\frac{1}{2}i, \frac{1}{2}j+1), & \text{if $j \in \{ i, i+2, \ldots, m \}$} \\
\mathscr{D}(\frac{1}{2}i, \frac{1}{2}(j-p)+1), & \text{if $j \in \{ m+p+2, m+p+4, \ldots, m+p+m+2 \}$} \\
\mathscr{D}(\frac{1}{2}i, l-\frac{1}{2}j), & \text{if $j \in \{ m+p+m+1, m+p+m-1, \ldots, m+p+1 \}$} \\
\mathscr{D}(\frac{1}{2}i, l-\frac{1}{2}(j+p)), & \text{if $j \in \{ m-1,m-3, \ldots, i+1 \}$}
\end{array}\right.
$$
and for $i \in \{ m+p+2, m+p+4, \ldots, m+p+m+2 \}$ one has
$$\arraycolsep=1.4pt\def\arraystretch{1.2}
d(\alpha) = \left\{\begin{array}{ll}
\mathscr{D}(\frac{1}{2}i, \frac{1}{2}j+1), & \text{if $j \in \{ i, i+2, \ldots, m+p+m+2 \}$} \\
\mathscr{D}(\frac{1}{2}i, l-\frac{1}{2}(j-5)), & \text{if $j \in \{ m+p+m+1, m+p+m-1, \ldots, i+1 \}$}.
\end{array}\right.
$$
This describes the image $d(\Delta^{O})$. For $m+2 \leq i \leq m+p+1$ one has
$$\arraycolsep=1.4pt\def\arraystretch{1.2}
d(\alpha) = \left\{\begin{array}{ll}
\mathscr{D}(m+p-1,j), & \text{if $j \in \{ m+p+2, m+p+4, \ldots, m+p+m+2 \}$} \\
\mathscr{D}(j,m+l+2), & \text{if $j \in \{ m+p+m+1, m+p+m-1, \ldots, m+p+1 \}$}.
\end{array}\right.
$$
This describes the image $d(\Delta^{C})$. For $m+1 \leq j \leq m+p$ one has
$$\arraycolsep=1.4pt\def\arraystretch{1.2}
d(\alpha) = \left\{\begin{array}{ll}
\mathscr{D}(i,m+3), & \text{if $i \in \{ m, m-2, \ldots, 2 \}$} \\
\mathscr{D}(-i,m+2), & \text{if $i \in \{ 1, 3, \ldots, m+1 \}$}.
\end{array}\right.
$$
This describes the image $d(\Delta^{R})$.

\subsection{} Case (iii). For $i \in \{ 2, 4, 6, \ldots, m+1 \}$ one has
$$\arraycolsep=1.4pt\def\arraystretch{1.2}
d(\alpha) = \left\{\begin{array}{ll}
\mathscr{D}(-\frac{1}{2}i, 2m-\frac{1}{2}j+3), & \text{if $j \in \{ i, i+2, \ldots, m-1 \}$} \\
\mathscr{D}(-\frac{1}{2}i,2m-\frac{1}{2}(j-p)+3), & \text{if $j \in \{ m+p+1, m+p+3, \ldots, m+p+m \}$} \\
\mathscr{D}(-\frac{1}{2}i,2m-l+\frac{1}{2}j+5), & \text{if $j \in \{ m+p+m+1, m+p+m-1, \ldots, m+p+2 \}$} \\
\mathscr{D}(-\frac{1}{2}i,2m-l+\frac{1}{2}(j+p)+5), & \text{if $j \in \{ m,m-2, \ldots, i+1 \}$}
\end{array}\right.
$$
and for $i \in \{ m+p+3, m+p+5, \ldots, m+p+m \}$ one has
$$\arraycolsep=1.4pt\def\arraystretch{1.2}
d(\alpha) = \left\{\begin{array}{ll}
\mathscr{D}(-\frac{1}{2}i, 2m-\frac{1}{2}j+3), & \text{if $j \in \{ i, i+2, \ldots, m+p+m+1 \}$} \\
\mathscr{D}(-\frac{1}{2}i, 2m-l+\frac{1}{2}(j+5)), & \text{if $j \in \{ m+p+m, m+p+m-2, \ldots, i+1 \}$}.
\end{array}\right.
$$
For $i \in \{ 1, 3, 5, \ldots, m \}$ one has 
$$\arraycolsep=1.4pt\def\arraystretch{1.2}
d(\alpha) = \left\{\begin{array}{ll}
\mathscr{D}(\frac{1}{2}i, \frac{1}{2}j+1), & \text{if $j \in \{ i, i+2, \ldots, m \}$} \\
\mathscr{D}(\frac{1}{2}i, \frac{1}{2}(j-p)+1), & \text{if $j \in \{ m+p+2, m+p+4, \ldots, m+p+m+1 \}$} \\
\mathscr{D}(\frac{1}{2}i, l-\frac{1}{2}j), & \text{if $j \in \{ m+p+m, m+p+m-2, \ldots, m+p+1 \}$} \\
\mathscr{D}(\frac{1}{2}i, l-\frac{1}{2}(j+p)), & \text{if $j \in \{ m-1,m-3, \ldots, i+1 \}$}
\end{array}\right.
$$
and for $i \in \{ m+p+2, m+p+4, \ldots, m+p+m+1 \}$ one has
$$\arraycolsep=1.4pt\def\arraystretch{1.2}
d(\alpha) = \left\{\begin{array}{ll}
\mathscr{D}(\frac{1}{2}i, \frac{1}{2}j+1), & \text{if $j \in \{ i, i+2, \ldots, m+p+m+1 \}$} \\
\mathscr{D}(\frac{1}{2}i, l-\frac{1}{2}(j-5)), & \text{if $j \in \{ m+p+m, m+p+m-2, \ldots, i+1 \}$}.
\end{array}\right.
$$
This describes the image $d(\Delta^{O})$. For $m+2 \leq i \leq m+p+1$ one has
$$\arraycolsep=1.4pt\def\arraystretch{1.2}
d(\alpha) = \left\{\begin{array}{ll}
\mathscr{D}(m+p-1,j), & \text{if $j \in \{ m+p+2, m+p+4, \ldots, m+p+m+1 \}$} \\
\mathscr{D}(j,m+l+3), & \text{if $j \in \{ m+p+m, m+p+m-2, \ldots, m+p+1 \}$}.
\end{array}\right.
$$
This describes the image $d(\Delta^{C})$. For $m+1 \leq j \leq m+p$ one has
$$\arraycolsep=1.4pt\def\arraystretch{1.2}
d(\alpha) = \left\{\begin{array}{ll}
\mathscr{D}(i,m+3), & \text{if $i \in \{ m, m-2, \ldots, 1 \}$} \\
\mathscr{D}(-i,m+2), & \text{if $i \in \{ 2, 4, \ldots, m+1 \}$}.
\end{array}\right.
$$
This describes the image $d(\Delta^{R})$. \\

\subsection{} Case (iv). For $i \in \{ 2,4,6, \ldots, m+1 \}$ one has
$$\arraycolsep=1.4pt\def\arraystretch{1.2}
d(\alpha) = \left\{\begin{array}{ll}
\mathscr{D}(-\frac{1}{2}i, 2m-\frac{1}{2}j+4), & \text{if $j \in \{ i, i+2, \ldots, m-1 \}$} \\
\mathscr{D}(-\frac{1}{2}i,2m-\frac{1}{2}(j-p)+4), & \text{if $j \in \{ m+p+1, m+p+3, \ldots, m+p+m+2 \}$} \\
\mathscr{D}(-\frac{1}{2}i,2m-l+\frac{1}{2}j+6), & \text{if $j \in \{ m+p+m+1, m+p+m-1, \ldots, m+p+2 \}$} \\
\mathscr{D}(-\frac{1}{2}i,2m-l+\frac{1}{2}(j+p)+6), & \text{if $j \in \{ m,m-2, \ldots, i+1 \}$}
\end{array}\right.
$$
and for $i \in \{ m+p+3, m+p+5, \ldots, m+p+m+2 \}$ one has
$$\arraycolsep=1.4pt\def\arraystretch{1.2}
d(\alpha) = \left\{\begin{array}{ll}
\mathscr{D}(-\frac{1}{2}i, 2m-\frac{1}{2}j+4), & \text{if $j \in \{ i, i+2, \ldots, m+p+m+2 \}$} \\
\mathscr{D}(-\frac{1}{2}i, 2m-l+\frac{1}{2}(j+7)), & \text{if $j \in \{ m+p+m+1, m+p+m-1, \ldots, i+1 \}$}.
\end{array}\right.
$$
For $i \in \{ 1, 3, 5, \ldots, m \}$ one has 
$$\arraycolsep=1.4pt\def\arraystretch{1.2}
d(\alpha) = \left\{\begin{array}{ll}
\mathscr{D}(\frac{1}{2}i, \frac{1}{2}j+1), & \text{if $j \in \{ i, i+2, \ldots, m \}$} \\
\mathscr{D}(\frac{1}{2}i, \frac{1}{2}(j-p)+1), & \text{if $j \in \{ m+p+2, m+p+4, \ldots, m+p+m+1 \}$} \\
\mathscr{D}(\frac{1}{2}i, l-\frac{1}{2}j), & \text{if $j \in \{ m+p+m+2, m+p+m, \ldots, m+p+1 \}$} \\
\mathscr{D}(\frac{1}{2}i, l-\frac{1}{2}(j+p)), & \text{if $j \in \{ m-1,m-3, \ldots, i+1 \}$}
\end{array}\right.
$$
and for $i \in \{ m+p+2, m+p+4, \ldots, m+p+m+1 \}$ one has
$$\arraycolsep=1.4pt\def\arraystretch{1.2}
d(\alpha) = \left\{\begin{array}{ll}
\mathscr{D}(\frac{1}{2}i, \frac{1}{2}j+1), & \text{if $j \in \{ i, i+2, \ldots, m+p+m+1 \}$} \\
\mathscr{D}(\frac{1}{2}i, l-\frac{1}{2}(j-5)), & \text{if $j \in \{ m+p+m+2, m+p+m, \ldots, i+1 \}$}.
\end{array}\right.
$$
This describes the image $d(\Delta^{O})$. For $m+2 \leq i \leq m+p+1$ one has
$$\arraycolsep=1.4pt\def\arraystretch{1.2}
d(\alpha) = \left\{\begin{array}{ll}
\mathscr{D}(m+p-1,j), & \text{if $j \in \{ m+p+2, m+p+4, \ldots, m+p+m+1 \}$} \\
\mathscr{D}(j,m+l+2), & \text{if $j \in \{ m+p+m+2, m+p+m, \ldots, m+p+1 \}$}.
\end{array}\right.
$$
This describes the image $d(\Delta^{C})$. For $m+1 \leq j \leq m+p$ one has
$$\arraycolsep=1.4pt\def\arraystretch{1.2}
d(\alpha) = \left\{\begin{array}{ll}
\mathscr{D}(i,m+3), & \text{if $i \in \{ m, m-2, \ldots, 1 \}$} \\
\mathscr{D}(-i,m+2), & \text{if $i \in \{ 2, 4, \ldots, m+1 \}$}.
\end{array}\right.
$$
This describes the image $d(\Delta^{R})$. \\

\end{document}